\documentclass[12pt]{article}

\usepackage[english]{babel}
\usepackage{amsfonts}
\usepackage{amssymb}
\usepackage{amsmath}
\usepackage{amsthm}
\usepackage{eucal}

\usepackage{enumerate}
\usepackage{upref}

\usepackage{graphicx}

\usepackage{hyperref}

\theoremstyle{plain}

\theoremstyle{plain}
\newtheorem{thm}{Theorem}[section]
\newtheorem{prop}[thm]{Proposition}
\newtheorem{lem}[thm]{Lemma}
\newtheorem{cor}[thm]{Corollary}
\newtheorem{sublem}[thm]{Sublemma}
\newtheorem{defn}[thm]{Definition}
\newtheorem{defs}[thm]{Definitions}
\newtheorem{example}[thm]{Example}

\newtheorem{rem}[thm]{Remark}

\theoremstyle{remark}

\numberwithin{equation}{section}

\newcommand{\set}[1]{\left\{#1\right\}}
\newcommand{\minus}{\smallsetminus}

\newcommand{\tree}{\mathsf{T}}
\newcommand{\F}{\mathsf{F}}
\newcommand{\VP}{\mathsf{v}^{\parent}}
\newcommand{\VC}{\mathsf{v}^{\child}}
\newcommand{\vc}{\mathsf{v}^{\child}}
\newcommand{\vp}{\mathsf{v}^{\parent}}

\newcommand{\bC}{\mathbb{C}}

\newcommand{\bN}{\mathbb{N}}
\newcommand{\bZ}{\mathbb{Z}}

\newcommand{\bdr}{\partial}
\newcommand{\dist}[2]{\textup{dist}\,(#1,#2)}

\DeclareMathOperator{\Ret}{Ret}

\newcommand{\child}{\textup{\textsf{c}}}
\newcommand{\parent}{\textup{\textsf{p}}}

\newcommand{\Kjulia}{\CMcal{K}}
\newcommand{\julia}{\CMcal{J}}

\newcommand{\esc}{\textup{esc}}
\DeclareMathOperator{\crit}{\textup{crit}}

\newcommand{\MR}[2]{\href{http://www.ams.org/mathscinet-getitem?mr=MR#1}{MR{#2}}}

\long\def\symbolfootnote[#1]#2{\begingroup\def\thefootnote{\fnsymbol{footnote}}\footnote[#1]{#2}\endgroup}

\begin{document}

\author{Nathaniel D. Emerson}

\title{On Yoccoz Return Functions}

\maketitle

\begin{center}

{Department of Mathematics\\ University of Southern California\\ Los Angeles, California 90089}


E-mail: \href{mailto:nemerson@usc.edu}{\texttt{nemerson@usc.edu}}

\end{center}

\begin{abstract}
We study the dynamics of complex polynomials. We obtain results on Poincar{\'e} return maps defined on certain neighborhoods of a
point with bounded orbit under a polynomial.  We introduce a generalization of the Yoccoz tau-function, the \emph{Yoccoz return
function}, which codes the returns of a critical point with bounded orbit of any complex polynomial with a disconnect Julia set.
We give necessary conditions on Yoccoz return functions, which allow for the recursive definition of an abstract tau-function.
These conditions are also sufficient for polynomials that have a disconnected Julia set and exactly one critical point with
bounded orbit.
\end{abstract}

\noindent \symbolfootnote[0]{\textit{Subject Classification.} Primary 37F10, 37F50, 37E25.}

\noindent \symbolfootnote[0]{\textit{Key Words and Phrases.} Julia set, tree with dynamics, Yoccoz tau-function.}

\tableofcontents

\section{Introduction}
Consider the dynamical system of a complex polynomial $f: \bC \to \bC$ of degree at least 2 (see \cite{CG} for example).  A point
is called \emph{persistent} if it has bounded orbit under $f$. Otherwise we say the point \emph{escapes (to infinity)}.  We
define $\Kjulia_f$, the \emph{filled Julia set of $f$}, as the set of all points that are persistent under $f$. The \emph{Julia
set of $f$}, $\julia_f$,  is the boundary of the filled Julia set.  A key question in determining the structure of the Julia set
of $f$ is the dynamics of the critical points of $f$. For example, the Julia of $f$ set is connected if and only if every
critical point of $f$ is persistent by a classical result of Fatou and Julia.

A polynomial with a unique critical point is called \emph{uni-critical}. We call a polynomial \emph{uni-persistent} if it has
exactly one persistent critical point (of any multiplicity). A uni-critical polynomial is uni-persistent if and only if its Julia
set is connected.  Some examples of a uni-persistent polynomial are a quadratic polynomial with a connected Julia set, or a cubic
polynomial with one critical point escaping and the other persistent. The dynamics of uni-critical polynomials have been widely
studied in these cases. The combinatorics of a uni-persistent polynomial are similar to the combinatorics of uni-critical
polynomial with a connected Julia set.

In this paper, we consider the dynamics of a persistent critical point of a polynomial. We are particularly interested in
uni-persistent polynomials with disconnected Julia sets. We code the dynamics of a polynomial using the combinatorial system of a
tree with dynamics \cite{E03}. We obtain results on various Poincar{\'e} return maps defined on certain neighborhoods of a
persistent point. For a uni-persistent polynomial with exactly one escaping critical point, the \emph{Yoccoz $\tau$-function}
\cite{Hubbard_Loc_Con} is a concise system for coding the returns of the persistent critical point. We introduce a generalization
of this function,  the \emph{Yoccoz return function} (Definition \ref{defn: tau}), which codes the returns of a persistent
critical point of any complex polynomial with a disconnected Julia set. We translate our results for return maps into necessary
conditions on the Yoccoz return function of a persistent critical point of a polynomial with a disconnected Julia set. These
conditions are recursively verifiable, so we can used them to define Yoccoz return functions abstractly. The conditions are
sufficient for a map on the integers to be realized as the Yoccoz return function of a uni-persistent polynomial with a
disconnected Julia set.  The following results are our main theorems. Let $\bN$ denote the non-negative integers.

\begin{thm} \label{main thm}
If $\tau: \bZ \to \bZ$ is the Yoccoz return function of a persistent critical point of a polynomial with a disconnected Julia
set, then there is an $H \in \bZ^+$ and an $E \subset \bN$ with $0 \in E$ such that the following condition hold for each $l\in
\bZ$:
\begin{enumerate}

    \item $\tau(l)= l-H$ if $l \leq H$, and $- H <\tau (l) < l $ if $l >H$;
    \item $\tau(l) = \tau^R(l-1)+1$ for some $R = R(l)  \geq 1$;
    \item if $\tau(l) = \tau^R(l-1)+1$ for some $R \geq 2$, then either
        \begin{enumerate}
            \item $\tau(\tau^{R-1}(l-1) + 1) \leq \tau^R(l-1)$,
            \item $\tau^{R-1}(l-1) \in E$.
        \end{enumerate}
\end{enumerate}
Moreover if the polynomial is uni-persistent, then $E$ is finite.

\end{thm}

We prove the converse of the above theorem for uni-critical polynomials. For a map $\tau : \bZ \to \bZ$, it is straightforward to
compute $\esc (\tau) \in \bZ^+$  (Definition \ref{defn: esc(m)}).

\begin{thm} \label{main thm converse}
Suppose $\tau: \bZ \to \bZ$ satisfies Conditions 1--3 above for some $H \in \bZ^+$ and some finite set $E\subset \bN$ with $0 \in
E$. For any integers $C \geq \esc(\tau)$ and $D \geq 2$, there is a uni-persistent polynomial $f$ of degree $C+D$ with a
disconnected Julia set such that
\begin{enumerate}
    \item the multiplicity the persistent critical point of $f$ is $D-1$;
    \item the Yoccoz return function of the persistent critical point of $f$ is $\tau$.
\end{enumerate}

\end{thm}

J.-C. Yoccoz introduced $\tau$-functions to study the dynamics of quadratic polynomials with connected Julia sets.  Some of his
results were published using the combinatorial system of tableaux \cite{Hubbard_Loc_Con}, which was developed by B.~Branner and
J.~Hubbard \cite[Prop.\ 4.1]{BH92}. For a uni-persistent polynomial with no more than one escaping critical point, the systems
are equivalent: a $\tau$-function defines a unique tableau and conversely a tableau defines a $\tau$-function \cite[Rem.\
9.3]{Hubbard_Loc_Con}. Branner and Hubbard gave 3 axioms for tableaux, which they claimed were necessary and sufficient
conditions for a tableau to be realizable as the tableau of a uni-persistent polynomial with exactly one escaping critical point
\cite[Prop.\ 12.8]{BH92}. In the notation of Theorem \ref{main thm}, their result is the case when $H=1$ and $E=\set{0}$. It was
later found that an additional axiom was needed in cubic case, see \S\ref{subsect: bi-critical returns}. We clarify which
polynomials require the fourth tableau axiom and which do not (Proposition \ref{prop: tab axioms}).  L.~DeMarco and A.~Schiff
\cite{DM-S-pp} translated the 4 tableaux axioms for cubic polynomials into the language of $\tau$-functions and independently
proved the equivalent of Theorems \ref{main thm} and \ref{main thm converse} for cubic polynomials.

Our main tool in this paper is the combinatorial system of a tree with dynamics (Definition \ref{defn: Tree Axioms}).  A tree
with dynamics was first used to study polynomials by R.~P{\'e}rez-Marco in an unpublished work \cite{PM_DCS}. Substantive results
using it were first obtained by the author \cite{E03}. A tree with dynamics shares properties of a tableau, but contains more
information.  The technical heart of this paper is Lemma \ref{main lem}, which is short and natural using trees with dynamics. It
is not clear how to express this lemma in terms of tableaux.

We prove our main theorems by considering the tree with dynamics associated to a polynomial with a disconnected Julia set
(\S\ref{subsect: Poly TwD}). A tree with dynamics encodes the key features of the dynamics of a polynomial. In particular, a
persistent critical point is encoded as a critical end of the tree (Definition \ref{defn: end of T}).  We consider various
Poincar{\'e} return maps defined on a tree with dynamics. Associated to each of these return maps is a set of vertices of the
tree, which we call \emph{portals} (Definition \ref{defn: X-portal}). Portals have the property that generally a vertex must be
iterated to a portal before it can return (Lemma \ref{main lem}). This property leads to necessary conditions on the first return
times of a critical end (Theorems \ref{thm: N1 cond} and Corollary \ref{cor: N1 cond for bi-critical}).  We translate these
conditions into the language of Yoccoz return functions, which give a more concise presentation of the combinatorics (Theorem
\ref{thm: tau cond}).  Thereby we prove Theorem \ref{main thm}. In order to show that the conditions in Theorem \ref{main thm
converse} are sufficient, we construct a tree with dynamics that realizes a given $\tau$ (Proposition \ref{prop: tau TwD}).
Similar conditions hold for polynomials with connected Julia sets and we note some results for connected Julia sets in various
places in this paper (Remarks \ref{rem: T1-T4 only => connected}, \ref{rem: Ret maps of connected} and \ref{rem: Main Lem for
connected}).

The remainder of this paper is organized as follows. The necessary technical background for this paper is given in Section
\ref{sect: TwD}. We outline the construction of a tree with dynamics of a polynomial (\S\ref{subsect: Poly TwD}).  We then give
axioms for abstract trees with dynamics and derive some basic properties of them (\S\ref{subsect: Abstract TwD}). In
\S\ref{subsect: 1st ret maps}, we consider various first return maps in a tree with dynamics. Section \ref{sect: portals} is the
heart of this paper. We define portals, prove our main lemma (Lemma \ref{main lem}), and state a version of Theorem \ref{main
thm} in terms of return maps (Theorem \ref{thm: N1 cond}). In \S\ref{subsect: bi-critical returns}, we study return maps and
classify portals for uni-persistent polynomials with exactly one escaping critical point. We classify portals and give a
corollary of Theorem \ref{main thm} for such polynomials (Corollary \ref{cor: N1 cond for bi-critical}). In \S\ref{subsect:
uni-persistent returns}, we consider general uni-persistent polynomials. We classify portals and prove a version of Theorem
\ref{main thm} for return maps (Theorem \ref{thm: N1 cond}). We consider Yoccoz return functions in \S\ref{sect: Yoccoz Ret}. We
prove some results about these functions, including Theorem \ref{main thm}, in \S\ref{subsect: properties of tau}. The conditions
in Theorem \ref{main thm} can be used to recursively define a Yoccoz return function.  We give some results about recursively
defining such a function, as well as some examples of such a definition in \S\ref{subsect: recursive defs of tau}. Finally in
Section \ref{sect: Realization}, we show Theorem \ref{main thm converse} can be realized by a polynomial.  We show this result by
constructing a tree with dynamics that realizes a specified $\tau$-function. The main steps of the proof are presented in
\S\ref{subsect: TwD from tau}, while the technical details are left until \S\ref{subsect: Proof of ext lem}.

\section{Trees with Dynamics}\label{sect: TwD}

This section contains the necessary background material for this paper, and is divided into 3 parts. First we briefly outline the
dynamic decomposition of the plane (\S\ref{subsect: Poly TwD}), which is used to define the tree with dynamics of a polynomial
with a disconnected Julia set. We then give axioms for an abstract tree with dynamics and recall some elementary properties
(\S\ref{subsect: Abstract TwD}). Finally we consider Poincar{\'e} return maps defined on a tree with dynamics (\S\ref{subsect:
1st ret maps}.

\subsection{The Tree with Dynamics of a Polynomial} \label{subsect: Poly TwD}

We define an \emph{annulus} as subset of the plane that is conformally equivalent to a set of the form $\set{z \in \bC: \ 0 \leq
r_1 < |z| < r_2 \leq \infty}$.  We say $S \subset \bC$ is \emph{nested} inside an annulus $A  $, if $S $ is contained in the
bounded components of $\bC \minus A$. For an annulus $A$, we define the \emph{filled-in annulus}:
\[
   P(A) = A \cup \set{\text{bounded components of } \bC \minus A}.
\]
Observe that $P(A)$ is an open topological disk.

Following Branner and Hubbard \cite{BH92}, we outline the dynamic decomposition of the plane. Fix a polynomial $f$ of degree
$d\geq 2$ with disconnected Julia set. Let $g$ denote Green's Function of $f$. The functional equation $g(f) = d \cdot g$ is
satisfied by $f$ and $g$. We use $g$ to define the dynamic decomposition of the basin of attraction of infinity for $f$.

An \emph{equipotential} is a level set of $g$: $\set{z \in \bC: \ g(z) = \text{const.}
>0}$.  By the functional equation, $f$ sends equipotentials to equipotentials. The
critical points of $g$ are the critical points of $f$ and the iterated pre-images of critical points of $f$. We distinguish all
equipotentials whose grand orbit contains a critical point of $f$. There are countably many such equipotentials, say
$\set{E_l}_{l \in \bZ}$. Index them so that $g|E_l < g|E_{l-1}$, $E_l$ is a Jordan curve for $l \leq 0 $, and $E_1$ is not a
Jordan curve (so it contains a subset homeomorphic to a figure-8). Let $H$ be the number of orbits of $\set{E_l}_{l \in \bZ}$
under $f$. If  $f$ has $e$ distinct critical points that escape to infinity, then $H \leq e$. It is possible that $H < e$, if $f$
has two escaping critical points $c$ and $c'$ such that $g(c) = d^n g(c')$ for some $n \in \bZ$. It follows that $f(E_l) =
E_{l-H}$ for any $l$ from the functional equation and the indexing of $E_l$.

Define $ U_l = \set{z: \ g|E_{l} > g(z) > g|E_{l+1}}$. For $l \leq 0$, $U_l$ is a single annulus.  For all $l$, $U_l$ is the
disjoint union of finitely many annuli $A_{l,i}$. We call each of the $A_{l,i}$ an \emph{annulus of $f$} at level $l$.  A
filled-in annulus of $f$, $P(A_{l,i})$, is a \emph{puzzle piece} of $f$ at level $l$ \cite{Branner}.  For any $A_{l,i}$, we have
$f(A_{l,i}) = A_{l-H,j}$ for some $j$. A sequence $(A_l)_{l \in \bZ}$ of annuli of $f$ is called \emph{nested}, if $A_l $ is at
level $l$ and $A_{l+1}$ is nested inside $A_l $ for all $l$. If $(A_l)$ is a nested sequence of annuli for a disconnected Julia
set, then $\bigcap_{l=0}^{\infty} P(A_l)$ is a component of $\Kjulia_f$. Thus there is a one-to-one correspondence between nested
sequence of annuli of $f$ and connected components of $\Kjulia_f$. We define the \emph{nest} of $z_0 \in \Kjulia_f$ as the nested
sequence of annuli of $f$, $(A_l)_{l\in \bN}$, such that $z_0$ is nested inside $A_l$ for all $l$. The \emph{extended nest} is
the analogous sequence with $l \in \bZ$.

The annuli of $f$ have a natural tree structure that is preserved by $f$.  We use them to define the tree with dynamics of $f$
(Definition \ref{defn: Tree Axioms}). We associate each annulus of $f$ to a vertex of the tree.  Let $A$ and $A'$ be annuli of
$f$ associated to vertices $\mathsf{v}$, $\mathsf{v}'$. Define an edge between $\mathsf{v}$ and $\mathsf{v}'$, if $A$ is nested
inside $A'$ and $\partial A \cap
\partial A' \neq \emptyset$.  In this case, we say $\mathsf{v}'$ is the \emph{parent} of
$\mathsf{v}$ (see Definition \ref{defn: tree}).  It can be shown that $f|A$ is a proper map, so it has a well-defined degree. We
define $\deg \mathsf{v}$ as the degree of $f|A$.

\begin{figure}[hbt]
    \begin{center}

    \includegraphics{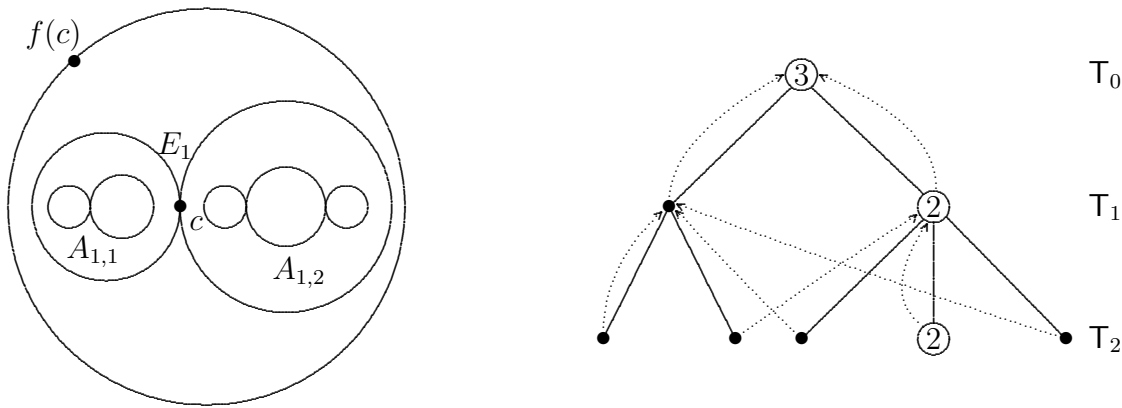}

    \caption{Equipotentials of a polynomial and the corresponding tree with dynamics.}

    \end{center}
\end{figure}

Since the tree with dynamics of $f$ is defined using $g$ and the functional equation, it is invariant under quasiconformal
homeomorphisms.

\subsection{Abstract Trees with Dynamics} \label{subsect: Abstract TwD}

We now give axioms for an abstract tree with dynamics. The tree with dynamics of any polynomial with a disconnected Julia set
satisfies these axioms \cite{E03}. We define a \emph{tree} as a countable connected graph with every circuit trivial. We say two
vertices of a graph are \emph{adjacent} if there is an edge between them. We only consider trees with a particular type of order
on their vertices.

\begin{defn} \label{defn: tree}
A \emph{genealogical tree} is a tree $\tree$ such that each vertex $\mathsf{v} \in \tree$ is associated with a unique adjacent
vertex $\vp$, the \emph{parent} of $\mathsf{v}$.  Every vertex adjacent to $\mathsf{v}$, except $\vp$, is called a \emph{child}
of $\mathsf{v}$ and denoted by $ \vc$.
\end{defn}

In this paper, by ``tree'' we mean genealogical tree.  We use the symbol $\tree$ to represent both the tree and its vertex set;
the edge set is left implicit. We use \textsf{sans serif} symbols for trees and objects associated with trees.  Our convention in
drawing trees is that a parent is above its children (see Fig.\ \ref{fig: TwD}). So $\VP$ is above $\mathsf{v}$ and any $\VC$ is
below $\mathsf{v}$.  When it is necessary to distinguish between children of $\mathsf{v}$ we use the notation
$\mathsf{v}^{\child_i}$. We say $\mathsf{v}$ is an \emph{ancestor} of $\mathsf{v}'$ if there are vertices $\mathsf{v}_0, \dots,
\mathsf{v}_n$ such that $\mathsf{v} = \mathsf{v}_0$, $\mathsf{v}' = \mathsf{v}_n$, and $\mathsf{v}_{i-1} =
\mathsf{v}_i^{\parent}$ for $i = 1, \dots, n$.  We say $\mathsf{v}''$ is a \emph{descendant} of $\mathsf{v}$ if $\mathsf{v}$ is
an ancestor of $\mathsf{v}''$.

\begin{defn} \label{defn: Tree Axioms}
We consider genealogical trees $\tree$ that satisfy the following axioms:

\begin{enumerate} [\indent(T1)]

    \item Each vertex has a unique parent.

    \item Each vertex has at least one, but only finitely many children.

   \item There is a distinguished vertex $\mathsf{v}_{0}$, the \emph{root} of $\tree$, with more than one child.

   \item There is a set of vertices $\set{\mathsf{v}_{-l}}_{l=1}^{\infty}$, such that $\set{\mathsf{v}_{-l}^{\child}} =
       \set{\mathsf{v}_{-l+1}}$ for all $l \in \bZ^+$.

   \end{enumerate}

\end{defn}

It follows that $\tree$ is locally finite, and has no leaves. The purpose of the set $\set{\mathsf{v}_{-l}}_{l=1}^{\infty} $ is
purely technical; it insures that all iterates of the dynamics (Definition \ref{defn: Dyn - abstract}) are defined.  The
important part of the tree is the subtree $\tree^* = \tree \minus \set{\mathsf{v}_{-l}}_{l=1}^{\infty}$.

\begin{defn}
Let $\tree$ be a genealogical tree.  We partition $\tree$ into \emph{levels} by defining $\tree_0 = \set{\mathsf{v}_{0}}$, and
recursively defining $\tree_l$ so that if $\mathsf{v} \in \tree_l$, then $\mathsf{v}^{\parent} \in \tree_{l-1}$ for any $l \in
\bZ$.
\end{defn}

We consider all infinite paths in the tree that move from parent to child.

\begin{defn} \label{defn: end of T}
Let $\tree$ be a tree.  An \emph{end} of ${\tree}$ is a sequence $\boldsymbol{\mathsf{x}}=(\mathsf{x}_l)_{l \in \bN}$, where
$\mathsf{x}_l \in \tree_l$ and $\mathsf{x}_{l-1} = \mathsf{x}_l^{\parent}$ for all $l$. An \emph{extended end} is the analogous
double sequence  $\boldsymbol{\mathsf{x}}=(\mathsf{x}_l)_{l \in \bZ}$.
\end{defn}

An end of $\tree$ corresponds to a nested sequence of annuli of a polynomial, which in turn corresponds to a connected component
of the filled Julia set of the polynomial. A natural metric for the extended ends of $\tree$ is a Gromov metric:
\[
    \dist{\boldsymbol{\mathsf{x}}}{\boldsymbol{\mathsf{y}}} = \gamma^{-L}, \quad
    L = \max \set{l \in \bZ: \ \mathsf{x}_l = \mathsf{y}_l},
\]
for some $\gamma > 1$.  Any two such metrics are equivalent.  We can extend such a metric to vertices of $\tree$ by taking the
minimum over all ends that contain the vertices. With respect to any of these metrics, the boundary of $\tree$ is the set of ends
union one point (corresponding to $\lim_{l \to - \infty} \mathsf{v}_l$). The boundary always has the topology of a Cantor set
union one isolated point.

The dynamics that we consider is a map on at tree that preserve the genealogical structure.

\begin{defn}\label{defn: Dyn - abstract}
Let $\tree$ be a tree.  A map $\F: \tree \to \tree$ \emph{preserves children} if for all $\mathsf{v}\in \tree$ the image of a
child of $\mathsf{v}$ is a child of $\F(\mathsf{v}) $.  Symbolically $\F(\vc) = \F(\mathsf{v})^{\child}$.

\end{defn}

A children-preserving map induces a well-defined map on the set of ends of the tree. Additionally such a map is continuous with
respect to any Gromov metric.  It is easy to check that if $\F: \tree \to \tree$ is a children-preserving map, then there exists
$H \in \bZ$ such that $\F(\tree_l) = \tree_{l-H}$ for all $l \in \bZ$.

\begin{rem} \label{rem: T1-T4 only => connected}
Everything in \S\ref{subsect: Abstract TwD} up to this point is true for a tree with dynamics of a polynomial with a connected
Julia set \cite{E08}. However in the tree with dynamics of a polynomial with a connected Julia set, there are finitely many
vertices where the degree is not defined.
\end{rem}

One can prove non-trivial results about polynomial dynamics by just considering a tree with a children-preserving map (see
\cite{E08}).  However we can prove more by keeping track of what we might call the ``polynomial-like'' structure of the dynamical
decomposition of the plane. A polynomial restricted to one of its annuli is a proper map (so it is polynomial-like
\cite{DH-Poly_like} on the filled annulus), therefore this restriction has a well-defined degree. We give axioms for an abstract
version of this degree.

\begin{defn} \label{defn: TwD}
A \emph{tree with dynamics $(\tree,\F)$} is a genealogical tree $\tree$, a children-preserving map $\F: \tree \to \tree$, and a
degree function $\deg : \tree \to \mathbb{Z}^+$, which satisfy the following axioms:
\begin{enumerate}[\indent(D1)]
  \item \label{eq: Monotone} (\emph{Monotonicity.})  For any $\mathsf{v} \in \tree$, we have
  \[
    \deg{\mathsf{v}} - 1 \geq \sum_{\set{\VC}} (\deg{\VC }   -1).
  \]

  \item \label{eq: LC} (\emph{Local Cover Property.})  For any $\mathsf{v} \in \tree$ and for each child
      $\F(\mathsf{v})^{\child_0}$ of $\F(\mathsf{v})$, we have
  \[
     \sum_{\set{\mathsf{v}^{\child}: \ \F(\mathsf{v}^{\child}) =  \F(\mathsf{v})^{\child_0}}} \deg{\mathsf{v}^{\child}} =
     \deg{\mathsf{v}}.
  \]

    \item \label{eq: deg extends} There exists $L\in \bZ$ such that for all $l \geq L$ if $\mathsf{v} \in \tree_l$, then
        $\deg \mathsf{v} = \deg_\mathsf{v} \F$, where $\deg_\mathsf{v} \F$ is the topological degree of $\F$ at $\mathsf{v}$
    \cite[Def.\ 4.3--4.5]{E03}.

    \item \label{eq: deg of root} For any $l < 0$, $ \deg \mathsf{v}_0 =  \deg \mathsf{v}_{l} $.
\end{enumerate}
\end{defn}

\begin{rem}
 It can be shown that the last 2 axioms are follow from the first 2.
\end{rem}

A more complete discussion of the axioms can be found in \cite{E03}. The first axiom says that $\mathsf{v}$ has no more critical
points inside it than its children. The second axiom says that $\F$ is locally a branched cover, and should be thought of as a
combinatorial Riemann-Hurwitz formula (for domains). Although technically a tree with dynamics is a triple $(\tree, \F, \deg)$,
for simplicity we usually denote it by the pair $(\tree, \F)$.

If $ \deg \mathsf{v}_0 = d$, then we say that $(\tree, \F)$ is a tree with dynamics \emph{of degree} $d$.  We say that
$\mathsf{v} \in \tree$ is \emph{critical} if $\deg \mathsf{v} > 0$.

Throughout this paper, $(\tree, \F)$ will denote tree with dynamics with $d = \deg \mathsf{v}_0 $ and $\F(\tree_l) = \tree_{l-H}$
for all $l\in \bZ$.

\begin{figure}[hbt] \label{fig: TwD}
    \begin{center}

    \includegraphics{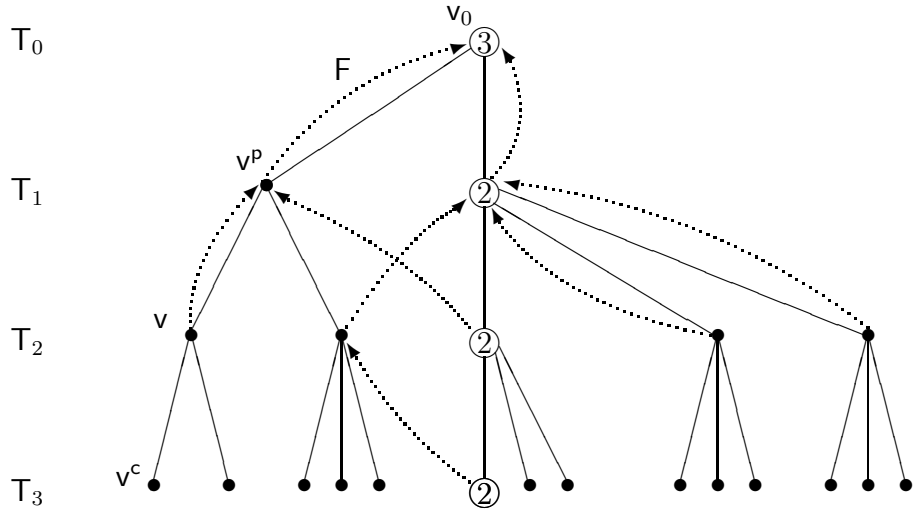}

    \caption{A tree with dynamics $(\tree, \F)$ of degree 3 with $H=1$. The critical vertices are marked with
    their degree, and the symbol $\bullet$ indicates a non-critical vertices. For clarity most of
    the dynamics from level 3 are not shown.}
    \end{center}
\end{figure}

We note some elementary properties of a tree with dynamics.

\begin{lem}\label{lem: TwD}
If $(\tree, \F)$ is a tree with dynamics, then the following hold.
\begin{enumerate}



    \item The degree function is \emph{monotone}:  if $\mathsf{v}'$ is a descendant of $\mathsf{v}$ for some $ \mathsf{v},
        \mathsf{v}' \in \tree$, then $\deg \mathsf{v}' \leq \deg \mathsf{v}$.\label{sublem: deg monotone}

      \item For any child $\mathsf{v}_0^{\child}$ of  $\mathsf{v}_0$, $\deg \mathsf{v}_0^{\child} < \deg \mathsf{v}_0
          $.\label{sublem: deg children of v_0}

    \item If $\mathsf{v} \in \tree$ and $\deg \mathsf{v}= \deg \mathsf{v}^{\child_0}$ for some child of $\mathsf{v}$, then
        every other child of $\mathsf{v}$ is non-critical. \label{sublem: deg v = deg vc => all other children n.c.}

\end{enumerate}

\end{lem}

\begin{lem}\label{lem: F^N a D-fold cover}
Let $\mathsf{v} \in \tree$ and let $\mathsf{w}= \F^N(\mathsf{v})$ for some $N \in \bZ^+$.  The map $\F^N:
\set{\mathsf{v}^{\child}} \to \set{\mathsf{w}^{\child}}$ is a $D$-fold cover where $D = \prod_{n=0}^{N-1}\deg \F^n(\mathsf{v})$.
That is, if $\mathsf{w}^{\child_0}$ is a child of $\mathsf{w}$, then
\[
    \sum_{\set{\vc: \ \F^N(\vc)= \mathsf{w}^{\child_0}  }} \deg \vc = D.
\]

\begin{proof}
The $N=1$ case is the local cover property (D2). The general case follows by induction on $N$.
\end{proof}
\end{lem}

\begin{defs}
Let $\mathsf{v} \in \tree $. We say that there is a \emph{split} at $\mathsf{v}$ if $\mathsf{v}$ has two critical children.  We
say that there is an \emph{escape} at $\mathsf{v}$ if $\deg{\mathsf{v}} - 1 > \sum_{\set{\VC}} (\deg{\VC }   -1)$.
\end{defs}

A simple manipulation of (D\ref{eq: Monotone}) gives the following result.

\begin{lem} \label{lem: esc or split iff deg drops}
There is an escape or a split at $\mathsf{v} \in \tree $ if and only if $ \deg{\mathsf{v}} > \max_{\set{\vc}} \deg \mathsf{v} $.
\end{lem}

\begin{cor}\label{cor: finitely many splits or escapes}
A tree with dynamics has only finitely many escapes or splits.
\begin{proof}
Let $m_l = \max \set{\deg \mathsf{v}: \ \mathsf{v} \in \tree_l}$. Then $(m_l)$ is a sequence of positive integers. It follows
(D\ref{eq: Monotone}) that it is non-increasing. Thus it can only decrease a finite number of times.
\end{proof}
\end{cor}

\begin{defs}
We define the degree of an (extended) end $\boldsymbol{\mathsf{x}}$ by $\deg \boldsymbol{\mathsf{x}} = \lim_{l \to \infty} \deg
\mathsf{x}_l$. If $\deg {\boldsymbol{\mathsf{x}}}
> 1$, then $\boldsymbol{\mathsf{x}}$ is called a \emph{critical end}.
\end{defs}

Immediately from (D\ref{eq: Monotone}) we see that the sequence $(\deg  \mathsf{x}_l)$ is non-increasing.

The number of critical ends of a tree with dynamics categorizes it in a manner analogous to our categorization of polynomials. We
call a tree with dynamics \emph{escaping} if it has no critical ends, \emph{uni-persistent} if it has exactly one critical end,
and \emph{multi-persistent} if it has at least two critical ends. The properties of a polynomial Julia set depend mainly on the
dynamics of the  critical points of the polynomial. So the key question about a tree with dynamics is what are the dynamics of
its critical ends?

\subsection{First Return Maps} \label{subsect: 1st ret maps}

We will study a tree with dynamics using Poincar{\'e} return maps. That is, a map which takes a point to its first iterate which
lies in some specified set. There are several sets which are obvious candidates to use as a target for a return map.  The set of
all critical vertices for instance (see \cite[Def.\ 5.5]{E03}). To begin with we do not choose a specific return map, but work
with the return map to an arbitrary subset of a tree with dynamics. Our definitions and basic results, in particular our Main
Lemma (Lemma \ref{main lem}), apply to any subset of a tree that contains all its ancestors. From \S\ref{subsect: bi-critical
returns} on, we only consider the Poincar{\'e} return map of a critical end.

We define the first return map to an arbitrary subset of a tree with dynamics.

\begin{defs} \label{defs: 1st Ret time and Map}
Let $(\tree, \F)$ be a tree with dynamics and let ${\mathsf{X}} \subset \tree$. We define the \emph{first return time of
$\mathsf{v} \in \tree$ to $\mathsf{X}$} by
\[
    N^1_{\mathsf{X}}(\mathsf{v}) = \min \set{n \geq 1: \ \F^n(\mathsf{v}) \in \mathsf{X}}
\]
provided $\F^n(\mathsf{v}) \in \mathsf{X} $ for some $ n \geq1$. We define the \emph{first return map to }$\mathsf{X}$ by
\[
    \Ret_{\mathsf{X}}(\mathsf{v})= \F^{N^1_{\mathsf{X}}(\mathsf{v})}(\mathsf{v}).
\]
For $R \geq 2$, we define the $R^{\text{th}}$ \emph{return time} of $\mathsf{v}$ to $\mathsf{X}$ by iterating the first return
map:
\[
    N^R_{\mathsf{X}}(\mathsf{v}) =
    N^1_{\mathsf{X}}(\Ret^{R-1}_{\mathsf{X}}(\mathsf{v})) .
\]
\end{defs}

If we iterate a vertex fewer times than its first return time, then the first return time is an additive function.

\begin{lem} \label{lem: ret times add}
Let $\mathsf{v} \in \tree$ with $N^1_{\mathsf{X}}(\mathsf{v})$ defined. If $1 \leq n < N^1_{\mathsf{X}}(\mathsf{v})$, then
\[
    N^1_{\mathsf{X}}(\mathsf{v}) = n + N^1_{\mathsf{X}}(\F^n(\mathsf{v})).
\]
\begin{proof}
Since $\F^{N^1_{\mathsf{X}}(\mathsf{v})}(\mathsf{v}) = \F^{N^1_{\mathsf{X}}(\mathsf{v})-n}(\F^n(\mathsf{v}))$, we have $
N^1_{\mathsf{X}}(\mathsf{v})-n = N^1_{\mathsf{X}}(\F^n(\mathsf{v})) $.
\end{proof}
\end{lem}

It is possible that some vertex never returns to a given set $\mathsf{X}$. We give a condition on $\mathsf{X}$ which insures that
the there is a return of any vertex to $\mathsf{X}$.

\begin{defn}
We say that ${\mathsf{X}} \subset \tree$ is \emph{ancestral} if for any $\mathsf{x} \in \mathsf{X}$, all ancestors of
$\mathsf{x}$ are elements of $\mathsf{X}$.
\end{defn}

It follows from the definition of an ancestor that a set $\mathsf{X}$ is ancestral if and only if $\mathsf{x}^{\parent} \in
\mathsf{X}$ whenever $\mathsf{x} \in \mathsf{X}$. An ancestral set is a subtree of $\tree$ possibly having leaves. Ancestral sets
are closed in $\overline{\tree} = \tree \cup \bdr \tree$. We leave the proofs of these facts as an exercise.

\begin{lem} \label{lem: ancestral => ret well defn}
If ${\mathsf{X}} \subset \tree$ is {ancestral}, then $N^1_{\mathsf{X}}(\mathsf{v})$ and $  \Ret_{\mathsf{X}}(\mathsf{v})$ are
well defined for any $\mathsf{v} \in \tree$.

\begin{proof}

Given $\mathsf{v} \in \tree$, we have $\mathsf{v}_l$ is an ancestor of $\mathsf{v}$ for some $l \leq 0$. Since $\mathsf{X}$ is
ancestral, $\mathsf{v}_l \in \mathsf{X}$.  We can find $n \geq 1 $ and $L \leq l$ such that $\F^n(\mathsf{v}) \in \tree_L$. Thus
$\mathsf{v}$ returns to $\mathsf{X}$, and so it must have a first return.

\end{proof}
\end{lem}

\begin{lem}
In any tree with dynamics $(\tree, \F)$, the following sets are ancestral:
\begin{enumerate}

    \item An extended end of $\tree$. \label{sublem: ends ancestral}

    \item The critical set of $\tree$. \label{sublem: critical set ancestral}

        \item Any union of ancestral subsets of $\tree$. \label{sublem: union of ancestral}
\end{enumerate}

\begin{proof}
An extended end is ancestral by Definition \ref{defn: end of T}. The critical set is ancestral by monotonicity of the degree
(Lemma \ref{lem: TwD}.\ref{sublem: deg monotone}). The proof of \ref{sublem: union of ancestral} is easy.

\end{proof}

\end{lem}

Henceforth we will only consider ancestral sets.  First return maps of ancestral sets are straightforward near the top of the
tree.

\begin{lem}\label{lem: N_1 T_l = 1 for l <H}
Suppose that ${\mathsf{X}} \subset \tree$ is {ancestral} and $\mathsf{v}_0 \in {\mathsf{X}}$.  If $\mathsf{v} \in \tree_l$ for
some $l \leq H$, then $ \Ret_{\mathsf{X}}(\mathsf{v}) = \mathsf{v}_{l-H}$ and $N^1_{\mathsf{X}}(\mathsf{v}) =1$.
\begin{proof}
Since $\mathsf{X}$ is {ancestral}, $\mathsf{v}_{l} \in \mathsf{X}$ for any $l \leq 0$. By assumption, $\F(\tree_l) =
\tree_{l-H}$.  For $l \leq H$, $\tree_{l-H} = \set{\mathsf{v}_{l-H}} $ by (T4). Thus $\F(\mathsf{v})    = \Ret_{\mathsf{X}}
(\mathsf{v})$ and $N^1_{\mathsf{X}}(\mathsf{v}) =1$.
\end{proof}

\end{lem}

A key property of ancestral sets is that when a vertex returns, it forces all of its ancestors to return.

\begin{lem}\label{lem: ret of descen => ret}
Suppose that ${\mathsf{X}} \subset \tree$ is {ancestral}.  Let $ \mathsf{v}, \mathsf{v}' \in \tree$ with $\mathsf{v}$ an ancestor
of $\mathsf{v}'$.  Any of return time of $\mathsf{v}'$ is some return time of $\mathsf{v}$; for any $S \geq 1$, there is $R \geq
S$ such that
\[
    N^S_{\mathsf{X}}(\mathsf{v}') = N^R_{\mathsf{X}}(\mathsf{v}).
\]
\begin{proof}
Let $N = N^S_{\mathsf{X}}(\mathsf{v}')$, so $\F^N(\mathsf{v}') =  \mathsf{x} $ for some $\mathsf{x} \in \mathsf{X}$. Since
${\mathsf{X}}$ is ancestral, $ \mathsf{x}^{\parent}  \in \mathsf{X}$. It follows by induction, that any ancestor of $
\F^N(\mathsf{v}')$ is a vertex of $\mathsf{X}$.  Since $\F$ is child preserving, we have that $\F^N (\mathsf{v}) $ is also in
$\mathsf{X}$. Thus $N = N^R_{\mathsf{X}}(\mathsf{v})$ for some $R \geq S$.
\end{proof}
\end{lem}

This has an important consequence for the returns to an end.

\begin{cor} \label{cor: N_1 non-dec }
Suppose that ${\mathsf{X}} \subset \tree$ is {ancestral}. For any extended end $\boldsymbol{\mathsf{x}} = (\mathsf{x}_l)_{l \in
\bZ}$, the sequence $(N^1_{{\mathsf{X}}}(\mathsf{x}_l))_{l\in \bZ}$ is non-decreasing.
\end{cor}

\begin{rem} \label{rem: Ret maps of connected}
All the results in this subsection hold for a tree with dynamics of a polynomial with a connected Julia set. Because they depend
only on the order property of the tree with dynamics (T1--T4), not the properties of the degree function (D1--D4).

\end{rem}


\section{Portals} \label{sect: portals}

In order to understand a return map on some ancestral subset of a tree with dynamics, the key question that we must answer is how
do the return times of a vertex restrict the first return time of its child? Given $\mathsf{v} \in \tree$ and some ancestral $X
\subset \tree$, for any child of $\mathsf{v}$ we have $N^1_{\mathsf{X}}(\mathsf{v}^{\child}) = N^R_{\mathsf{X}}(\mathsf{v}) $ for
some $R \in \bZ^+$ by Lemma \ref{lem: ret of descen => ret}. It is not hard to derive an upper bound $M$ for $R$ using Lemma
\ref{lem: N_1 T_l = 1 for l <H} (see Corollary \ref{cor: bounds on R(l)}).  So we can say that $R \in \set{1, \dots, M}$.
However not all of these values can occur. Specifically Lemma \ref{main lem} (our Main Lemma) shows that the child cannot return
until some iterate of its parent lies in a set of vertices that we call ``portals''. We say the parent ``passes through'' the
portal. The Main Lemma is based on a new insight on the behavior of return maps. It is easy to prove, but has important
consequences.

In order to make the above results useful, we must classify portals in terms of the behavior of a return map.  We can classify
all portals in a tree with dynamics for a natural type of return map. We consider the return map to a critical end in this paper.
We classify portals for two classes of trees with dynamics. Our results completely describe return maps for uni-persistent trees
with dynamics.

In \S\ref{subsect: Main Thm for N1} we define portals.  We prove our Main Lemma.  We then start to classify various types of
portals. We consider bi-critical polynomials (that is, polynomials with exact two distinct critical points) in \S\ref{subsect:
bi-critical returns}.  We classify portals for bi-critical polynomials and obtain a corollary of Theorem \ref{main thm} for
bi-critical polynomials (Corollary \ref{cor: N1 cond for bi-critical}).  This corollary is essentially equivalent to a result of
Branner and Hubbard for tableaux \cite[Prop.\ 12.8]{BH92}, and Proposition \ref{prop: tab axioms} clarifies their result. Finally
we consider uni-persistent polynomials in \S\ref{subsect: uni-persistent returns}. Again we classify portals in this case, and
derive necessary and sufficient conditions on return times (Theorem \ref{thm: N1 cond}).

We start with the fundamental definition of this section.  Throughout this section, let $(\tree, \F)$ be some tree with dynamics.

\begin{defn} \label{defn: X-portal}
Let $\mathsf{X} \subset \tree$. We say that $\mathsf{x} \in \mathsf{X}$ is an \emph{$\mathsf{X}$-portal} if $\mathsf{x}$ has a
child $\mathsf{x}^{\child_0} \notin \mathsf{X}$ such that $N^1_{\mathsf{X}}(\mathsf{x}^{\child_0}) =
N^1_{\mathsf{X}}(\mathsf{x})$.
\end{defn}

A vertex must \emph{pass through} an $\mathsf{X}$-portal before it returns--hence the name.

\begin{lem}[Main Lemma]\label{main lem}
Let $\mathsf{v}\in \tree$ such that $N^1_{\mathsf{X}}(\mathsf{v}^{\child_0}) > N^1_{\mathsf{X}}(\mathsf{v})$ for some child of
$\mathsf{v}$. If $N^1_{\mathsf{X}}(\mathsf{v}^{\child_0}) =N^R_{\mathsf{X}}(\mathsf{v})$ for some $R \geq 2$, then
$\F^{N^{R-1}_{\mathsf{X}}(\mathsf{v})}(\mathsf{v})$ is an $\mathsf{X}$-portal.

\begin{proof}
Say that $\F^{N^{R-1}_{\mathsf{X}}(\mathsf{v})}(\mathsf{v}) = \mathsf{w}$ and
$\F^{N^{R-1}_{\mathsf{X}}(\mathsf{v})}(\mathsf{v}^{\child_0}) = \mathsf{w}^{\child_0}$. By definition of
$N^{R-1}_{\mathsf{X}}(\mathsf{v})$, $\mathsf{w} \in \mathsf{X}$. Since $N^1_{\mathsf{X}}(\mathsf{v}^{\child_0})
=N^R_{\mathsf{X}}(\mathsf{v}) > N^{R-1}_{\mathsf{X}}(\mathsf{v})$, $\mathsf{w}^{\child_0} \notin \mathsf{X}$. It follows from
Lemma \ref{lem: ret times add} that $N^1_{\mathsf{X}}(\mathsf{w}^{\child_0}) = N^1_{\mathsf{X}}(\mathsf{w})$.
\end{proof}

\end{lem}

If ${\mathsf{X}}$ is {ancestral}, then  $N^1_{\mathsf{X}}(\mathsf{v}^{\child_0}) =N^R_{\mathsf{X}}(\mathsf{v})$ for some $R \geq
1$ by Lemma \ref{lem: ret of descen => ret}.


Our Main Lemma shows that considering the tree with dynamics can lead to insights that would be difficult to see using tableaux.
Consider the iterates of some end $\boldsymbol{\mathsf{x}} = (\mathsf{x}_l)_{l=0}^{\infty}$.  The tableau of
$\boldsymbol{\mathsf{x}}$ is the marked grid $\set{a_{m,n}}  = \F^n(\mathsf{x}_{m+n})$ for $m,n \geq 0$ (with critical vertices
marked) \cite{BH92}. So a tableau is a sequence of ends.  In particular, a tableau does not keep track of children of some
$\F^n(\mathsf{x}_{m+n})$ that are not in $\F^n( \boldsymbol{\mathsf{x}} )$.  The definition of portal and the Main Lemma are
based on consideration of such children.

The Main Lemma shows that we can study the first return map by studying $\mathsf{X}$-portals. We can extend it to higher return
times by induction.

\begin{cor} \label{main cor}
Let $\mathsf{v} \in \tree$ and let $\mathsf{v}^{\child_0}$ be a child of $\mathsf{v}$. Suppose that $N_{\mathsf{X}}^R(\mathsf{v})
= N_{\mathsf{X}}^S(\mathsf{v}^{\child_0})$ for some $1 \leq S \leq R$ with $R\geq 2$. If
$\F^{N_{\mathsf{X}}^{R-1}(\mathsf{v})}(\mathsf{v}^{\child_0}) \notin \mathsf{X}$, then
$\F^{N_{\mathsf{X}}^{R-1}(\mathsf{v})}(\mathsf{v})$ is an $\mathsf{X}$-portal.
\end{cor}

\begin{rem} \label{rem: Main Lem for connected}
The Main Lemma and Corollary \ref{main cor} hold for a tree with dynamics of a polynomial with a connected Julia set. It is not
clear if the remainder of the results in this section hold, since they are largely dependent on careful analysis of the degree
function.
\end{rem}

For the remainder of this paper, we will only consider first return maps in the case when $\mathsf{X} = \boldsymbol{\mathsf{c}}$
a critical extended end.

We obtain the following version of Theorem \ref{main thm} for return times.  We defer the proof until \S\ref{subsect: Main Thm
for N1}.

\begin{thm} \label{thm: N1 cond}
Let $(\tree, \F)$ be a tree with dynamics.  Let $\boldsymbol{\mathsf{c}} = (\mathsf{c}_l)_{l \in \bZ}$ be a critical extended end
of $\tree$. The following conditions hold for each $\mathsf{v} \in \tree$:

\begin{enumerate}
    \item If $\mathsf{v} \in \tree_l$, then $ N^1_{\mathsf{\boldsymbol{\mathsf{c}}}}(\mathsf{c}_l) =1$ if $l \leq H$ and $1
        \leq N^1_{\mathsf{\boldsymbol{\mathsf{c}}}}(\mathsf{c}_l) \leq \lceil l/H \rceil $ if $l > H$;

    \item For any child of $\mathsf{v}$, $N^1_{\mathsf{\boldsymbol{\mathsf{c}}}}(\mathsf{v}^{\child}) =
        N^R_{\mathsf{\boldsymbol{\mathsf{c}}}}(\mathsf{v})$ for some $R \geq 1$;

    \item If $N^1_{\mathsf{\boldsymbol{\mathsf{c}}}}(\mathsf{v}^{\child}) =
        N^R_{\mathsf{\boldsymbol{\mathsf{c}}}}(\mathsf{v})$ for some $R \geq 2$, then one of the following conditions holds:

    \begin{enumerate}
        \item $\Ret^{R-1}_{\boldsymbol{\mathsf{c}}}(\mathsf{v})= \mathsf{c}_k$ for some $k$ such that
            $N^1_{\boldsymbol{\mathsf{c}}}(\mathsf{c}_{k+1} )> N^1_{\boldsymbol{\mathsf{c}}}(\mathsf{c}_k)$;

        \item $\Ret^{R-1}_{\boldsymbol{\mathsf{c}}}(\mathsf{v}) \in \crit(\tree) \minus \boldsymbol{\mathsf{c}}$;

        \item there is an escape or split at $\Ret^{R-1}_{\boldsymbol{\mathsf{c}}}(\mathsf{v})$.
    \end{enumerate}

\end{enumerate}

\end{thm}


The above conditions are also sufficient for uni-persistent polynomials, see \S\ref{sect: Realization}.

\subsection{Main Theorem for Return Times} \label{subsect: Main Thm for N1}

Part of what makes our Main Lemma useful is that we can classify where portals occur entirely in terms of the first return map.
We classify the portals of various return maps.

\begin{defn} \label{defn: Types of portals}
Suppose that ${\mathsf{X}} \subset \tree$ is {ancestral}.  Let $\mathsf{x}$ be an $\mathsf{X}$-portal. We classify the
\emph{type} of $\mathsf{x}$ as follows.

\begin{enumerate} [{\indent}I.]
    \item If $\deg \mathsf{x} = \deg \mathsf{x}^{\child_1} $ and $N^1_{\mathsf{X}}(\mathsf{x}^{\child_1})
> N^1_{\mathsf{X}}(\mathsf{x})$ for some child of $\mathsf{x}$.

      \item  If $\deg \mathsf{x} = \deg \mathsf{x}^{\child_1} $ and $N^1_{\mathsf{X}}(\mathsf{x}^{\child_1}) =
          N^1_{\mathsf{X}}(\mathsf{x})$ for some child of $\mathsf{x}$.

   \item  If $\deg \mathsf{x} > \deg \mathsf{x}^{\child} $ for every child of $\mathsf{x}$.
\end{enumerate}

We define $\textup{Port}_i({\mathsf{X}})$ as the set of all type $i$ $\mathsf{X}$-portals for $i =I,II,III$.

\end{defn}

First we give some results about general $\mathsf{X}$-portals.  We find conditions on where portals of various types can occur. A
basic result is that we should look for type~I portals where the return time increases.

\begin{lem} \label{lem: catagorize uni-per I portals}
Let $(\tree, \F)$ be a uni-persistent tree with dynamics. Let $\boldsymbol{\mathsf{c}} = (\mathsf{c}_l)_{l \in \bZ}$ be the
(unique) critical end of $\tree$. Suppose that $\mathsf{X}$ is an ancestral set with $\boldsymbol{\mathsf{c}} \subset \mathsf{X}
\subset \crit(\tree)$. For each $l \in \bZ$, $\mathsf{c}_l \in \textup{Port}_{\textup{I}}(\mathsf{X})$ if and only if
$N^1_{\mathsf{X}}(\mathsf{c}_l)< N^1_{\mathsf{X}}(\mathsf{c}_{l+1})$ and $\deg \mathsf{c}_{l} = \deg \mathsf{c}_{l+1}$.
\begin{proof}
Suppose that $\mathsf{c}_l \in \textup{Port}_{\textup{I}}(\mathsf{X})$. By Definition \ref{defn: Types of portals}, $\deg
\mathsf{c}_l = \deg \mathsf{c}_l^{\child_1} $ and $ N^1_{\mathsf{X}}(\mathsf{c}_l) < N^1_{\mathsf{X}}(\mathsf{c}_l^{\child_1})$
for some child of $\mathsf{c}_l$. Since the critical set of $\tree$ is $\mathsf{X}$, we have $ \mathsf{c}_l^{\child_1} =
\mathsf{c}_{l+1}$.

Conversely suppose that $N^1_{\mathsf{X}}(\mathsf{c}_l)< N^1_{\mathsf{X}}(\mathsf{c}_{l+1})$ and $\deg \mathsf{c}_{l} = \deg
\mathsf{c}_{l+1}$ for some $l \in \bZ$. Say that $\Ret_{\mathsf{X}}(\mathsf{c}_{l}) = \F^N(\mathsf{c}_{l}) = \mathsf{c}_{k}$.  By
assumption, $\F^N(\mathsf{c}_{l+1}) \neq \mathsf{c}_{k+1}$. By Lemma \ref{lem: F^N a D-fold cover}, there is a child of
$\mathsf{c}_{l}$ such that $\F^N(\mathsf{c}_l^{\child_1}) = \mathsf{c}_{k+1} $.  Since $\mathsf{c}_l^{\child_1} \neq
\mathsf{c}_{l+1}$, so $\mathsf{c}_l^{\child_1} \notin \mathsf{X} $. Therefore $\mathsf{c}_l \in
\textup{Port}_{\textup{I}}(\mathsf{X})$.
\end{proof}
\end{lem}

Type~II portals are more complicated.

\begin{lem} \label{lem: port II cond}
Let ${\mathsf{X}} \subset \tree$.  If $\mathsf{x}\in \textup{Port}_{\textup{II}}({\mathsf{X}})$, then at least one of the
following conditions holds:

\begin{enumerate}
    \item There is a number $n$ with $1 \leq n < N^1_{\mathsf{X}} (\mathsf{x})$ such that $\F^n(\mathsf{x})$ is critical and
        $\F^n(\mathsf{x})$ has a child, $\F^n(\mathsf{x})^{\child} \notin X$ such that
        \[
            N^1_{\mathsf{X}}(\F^n(\mathsf{x})^{\child})= N^1_{\mathsf{X}}(\F^n(\mathsf{x})) =  N^1_{\mathsf{X}}(\mathsf{x}) - n .
        \]
    \item At least two distinct children of $\Ret_{\mathsf{X}}(\mathsf{x})$ are members of $\mathsf{X}$.

\end{enumerate}

\end{lem}

Case 2 of the above lemma cannot occur when $\mathsf{x}$ is an end.

\begin{cor}\label{cor: Type II portals of an end}
Let $\boldsymbol{\mathsf{x}}  = (\mathsf{x}_{l})$ be an extended end of $\tree$.  If $\mathsf{x}\in
\textup{Port}_{\textup{II}}(\boldsymbol{\mathsf{x}} )$, then case 1 of Lemma \ref{lem: port II cond} holds.
\begin{proof}
By definition, no vertex of an end can have two distinct children that are in the end. So case 2 above never occurs.
\end{proof}

\end{cor}

Type~III portals are straightforward to categorize. We obtain the following two results immediately from Lemma \ref{lem: esc or
split iff deg drops} and Corollary \ref{cor: finitely many splits or escapes}

\begin{lem}\label{lem: Port III categorized}
Let ${\mathsf{X}} \subset \tree$.  If $\mathsf{x} \in \textup{Port}_{\textup{III}}({\mathsf{X}})$, then there is an escape or
split at $\mathsf{x}$.

\end{lem}

\begin{cor} \label{cor: finite number of type 3 portals}
For any ${\mathsf{X}} \subset \tree$, $\textup{Port}_{\textup{III}}({\mathsf{X}})$ is finite.
\end{cor}

\begin{lem} \label{lem: v_0 a type 3 portal}
If $\boldsymbol{\mathsf{x}}$ is an end of some tree with dynamics, then $\mathsf{v}_0 \in
\textup{Port}_{\textup{III}}(\boldsymbol{\mathsf{x}})$.

\begin{proof}
By (T3), $\mathsf{v}_0 = \mathsf{x}_0$ has at least two children.  Because $ \boldsymbol{\mathsf{x}}$ is an end $\mathsf{v}_0$
has a child $\mathsf{v}_0^{\child_0} \notin \boldsymbol{\mathsf{x}}$. By Lemma \ref{lem: N_1 T_l = 1 for l <H},
$N^1_{\boldsymbol{\mathsf{x}}}(\mathsf{v}_{0}) = N^1_{\boldsymbol{\mathsf{x}}}(\mathsf{v}_{0}^{\child_0}) = 1$. Therefore
$\mathsf{v}_0$ is a $\boldsymbol{\mathsf{x}}$-portal. By Lemma \ref{lem: TwD}.\ref{sublem: deg children of v_0}, $\deg
\mathsf{v}_0 > \deg\mathsf{v}_0^{\child_0} $. Hence $\mathsf{v}_0\in \textup{Port}_{\textup{III}}(\boldsymbol{\mathsf{x}})$.

\end{proof}
\end{lem}

It is important to note how many children an $\mathsf{X}$-portals has that satisfy the conditions of Definition \ref{defn:
X-portal}.

\begin{defn}
Let $\mathsf{x}$ be an $\mathsf{X}$-portal.  We call $\mathsf{x}$  a \emph{simple} $\mathsf{X}$-portal if $\mathsf{x}$ has a
unique child $\mathsf{x}^{\child_0} \notin \mathsf{X}$ with $N^1_{\mathsf{X}}(\mathsf{x}^{\child_0}) =
N^1_{\mathsf{X}}(\mathsf{x})$. Otherwise we call $\mathsf{x}$  a \emph{compound} $\mathsf{X}$-portal.
\end{defn}

The following lemma allows us to determine whether type~III $\boldsymbol{\mathsf{c}}$-portal is simple or compound.

\begin{lem}\label{lem: number of n.c.c of type 3 c-portal}
Let $\boldsymbol{\mathsf{c}} = (\mathsf{c}_{l})_{l \in \bZ} $ be a critical extended end of $(\tree, \F)$.  Suppose that $
N^1_{\boldsymbol{\mathsf{c}}}(\mathsf{c}_l) =  N^1_{\boldsymbol{\mathsf{c}}}(\mathsf{c}_{l+1})$ for some $l$. If there is not a
split at $\mathsf{c}_{l}$, then $\mathsf{c}_{l}$ has at least $\deg \mathsf{c}_{l} - \deg \mathsf{c}_{l+1}$ non-critical children
$\mathsf{c}_{l}^{\child}$ with $N^1_{\boldsymbol{\mathsf{c}}}(\mathsf{c}_l^{\child}) =
N^1_{\boldsymbol{\mathsf{c}}}(\mathsf{c}_l) $.
\begin{proof}
Since $\mathsf{c}_{l+1} \in \boldsymbol{\mathsf{c}}$, $\mathsf{c}_{l+1}$ is critical. By assumption there is not a split at
$\mathsf{c}_{l}$, so $\mathsf{c}_{l+1}$ is the only critical child of $\mathsf{c}_{l}$. By Lemma \ref{lem: F^N a D-fold cover},
$\mathsf{c}_{l}$ has at least $\deg \mathsf{c}_{l}$ children counted by degree with
$N^1_{\boldsymbol{\mathsf{c}}}(\mathsf{c}_l^{\child}) = N^1_{\boldsymbol{\mathsf{c}}}(\mathsf{c}_l) $. Therefore there are at
least  $\deg \mathsf{c}_{l} - \deg \mathsf{c}_{l+1}$ of these children which are non-critical (which might be 0 of them).
\end{proof}
\end{lem}

\begin{cor} \label{cor: c-l simple iff deg c-1 - deg c-l+1 =1}
Let $\boldsymbol{\mathsf{c}} = (\mathsf{c}_{l})_{l \in \bZ} $ be a critical extended end of $(\tree, \F)$.  Suppose that $
N^1_{\boldsymbol{\mathsf{c}}}(\mathsf{c}_l) =  N^1_{\boldsymbol{\mathsf{c}}}(\mathsf{c}_{l+1})$ for some $l$. If there is an
escape but not a split at $\mathsf{c}_{l}$, then $\mathsf{c}_{l}$ is a simple $\boldsymbol{\mathsf{c}}$-portal if and only if
$\deg \mathsf{c}_{l} - \deg \mathsf{c}_{l+1} = 1$.
\end{cor}

\subsection{Bi-Critical Polynomials} \label{subsect: bi-critical returns}

We call a polynomial \emph{bi-critical} if it has exactly two critical points (of any multiplicities). For instance, a generic
cubic polynomial is bi-critical. In this subsection, we consider the tree with dynamics of a bi-critical polynomial with
disconnected Julia set. If both of the critical points escape to infinity, then the Julia set is an area zero Cantor set.  So the
interesting case is when one critical point escapes and the other is persistent. Combinatorially we ignore the trivial case and
call a tree with dynamics \emph{bi-critical} if it is uni-persistent and there is exactly one vertex in the tree where there is
an escape.

Bi-critical polynomials are the easiest class of polynomials to analyze using the tree with dynamics. There two other main
reasons that they are important. First, it is only class for which conditions on the tableaux are known.  Branner and Hubbard's
three tableaux axioms \cite{BH92} are necessary for this class of polynomial.  The three tableaux axioms combined with a fourth
axiom \cite{Kiwi_Puiseux} are known to be necessary and sufficient in the cubic case.

Second a uni-persistent polynomial $f$ can be associated to a bi-critical tree with dynamics in the following manner: choose
equipotentials $E_0$ and $E_1$ of $f$ such that all critical points of $f$ lie in the bounded component of $\bC \minus E_0$, all
escaping critical points of $f$ lie in the closure of the unbounded component of $\bC \minus E_1$, and $f^n(E_1)= E_0$ for some
$n \geq 1$. Extend to a collection of equipotentials $\set{E_l}_{l\in \bZ}$ using the relationship $f^n(E_l) = E_{l-1}$ for each
$l \in \bZ$.  Form the annuli and  then tree with dynamics of $f$ using these equipotential as in \S\ref{subsect: Poly TwD}.
Equivalently we can extract a subtree with dynamics from $(\tree, f^n)$  with $H=1$ using \cite[Prop.\ 5.3]{E03}.

Part of what makes the analysis of the bi-critical case easy is that the degree function only changes once on the critical end.

\begin{lem} \label{lem: deg fnc of uni-per}
Let $(\tree, \F)$ be a bi-critical tree with dynamics of degree $d$. If $\boldsymbol{\mathsf{c}} = (\mathsf{c}_l)$ is the
(unique) critical end of $\tree$, then
\[
    \deg  \mathsf{c}_l = D \  (l \geq 1) \quad \text{and} \quad
       \deg  \mathsf{c}_l = d \  (l \leq 0)
\]
for some $D$ with $2 \leq D <d$.
\begin{proof}

The unique escape in a bi-critical tree with dynamics must occur at $\mathsf{v}_0 = \mathsf{c}_0$ by Lemma \ref{lem: TwD}.2.
\end{proof}
\end{lem}

For the remainder of this section, let $\boldsymbol{\mathsf{c}} = (\mathsf{c}_l)_{l \in\bZ}$ be the unique critical end of a
uni-persistent tree with dynamics. We consider the first return map to $\boldsymbol{\mathsf{c}}$. So for any $\mathsf{v} \in
\tree$, Definition \ref{defs: 1st Ret time and Map} gives
\[
    N^1_{\boldsymbol{\mathsf{c}}}(\mathsf{v}) = \min \set{n \geq 1: \ \F^n(\mathsf{v}) \in \boldsymbol{\mathsf{c}}}
    \quad \text{and} \quad
    \Ret_{\boldsymbol{\mathsf{c}}}(\mathsf{v})= \F^{N^1_{\boldsymbol{\mathsf{c}}}(\mathsf{v})}(\mathsf{v}).
\]

Note that $\mathsf{v}_l = \mathsf{c}_l$ for $l \leq 0$. Thus for any $\mathsf{v} \in \tree$,
$N^R_{\boldsymbol{\mathsf{c}}}(\mathsf{v})$ and $\Ret^R_{\boldsymbol{\mathsf{c}}}(\mathsf{v})$ are well defined for all $R \geq
1$. We are primarily interested in the restriction of $ N^1_{\boldsymbol{\mathsf{c}}}$ to $\boldsymbol{\mathsf{c}}$.

We classify $\boldsymbol{\mathsf{c}}$-portals in the bi-critical case. The classification is straightforward.  Type I
$\boldsymbol{\mathsf{c}}$-portals occur when the return time increases, Type II $\boldsymbol{\mathsf{c}}$-portals do not occur,
and $\mathsf{c}_0$ is the only type III $\boldsymbol{\mathsf{c}}$-portal. First we classify type I
$\boldsymbol{\mathsf{c}}$-portals using Lemma \ref{lem: catagorize uni-per I portals}.

\begin{lem} \label{lem: bi-critical type I portals}
Let $(\tree, \F)$ be a bi-critical tree with dynamics with critical end $\boldsymbol{\mathsf{c}} = (\mathsf{c}_l)_{l \in \bZ}$.
For each $l \geq 2$, $\mathsf{c}_l \in \textup{Port}_{\textup{I}}(\boldsymbol{\mathsf{c}})$ if and only if
$N^1_{\boldsymbol{\mathsf{c}}}(\mathsf{c}_l)< N^1_{\boldsymbol{\mathsf{c}}}(\mathsf{c}_{l+1})$.

\end{lem}

\begin{cor} \label{cor: classify type I c-portals}
Let $(\tree, \F)$ be a bi-critical tree with dynamics with critical end $\boldsymbol{\mathsf{c}} = (\mathsf{c}_l)_{l \in \bZ}$.
If $\mathsf{c}_l \in \textup{Port}_{\textup{I}}(\boldsymbol{\mathsf{c}})$ for some $l \in \bZ$, then $\mathsf{c}_l$ is a compound
$\boldsymbol{\mathsf{c}}$-portal.
\begin{proof}
By Lemma \ref{lem: catagorize uni-per I portals}, $N^1_{\boldsymbol{\mathsf{c}}}(\mathsf{c}_l)  <
N^1_{\boldsymbol{\mathsf{c}}}(\mathsf{c}_{l+1})$. It follows from Lemma \ref{lem: F^N a D-fold cover} that $\mathsf{c}_l$ has
exactly $\deg \mathsf{c}_l$ children (counted by degree) such that $N^1_{\boldsymbol{\mathsf{c}}}(\mathsf{c}_l^{\child_i})=
N^1_{\boldsymbol{\mathsf{c}}}(\mathsf{c}_l)$.  None of these children are critical sine the tree with dynamics is bi-critical, so
there are at least two of them.
\end{proof}
\end{cor}

Next we show that there are no type~II $\boldsymbol{\mathsf{c}}$-portals in the bi-critical case.

\begin{lem} \label{lem: no bi-critical II portals}
Let $(\tree, \F)$ be a uni-persistent tree with dynamics with critical end $\boldsymbol{\mathsf{c}}$. If $\crit(\tree) =
\boldsymbol{\mathsf{c}}$, then $\textup{Port}_{\textup{II}}(\boldsymbol{\mathsf{c}}) = \emptyset$.
\begin{proof}
Suppose that $\mathsf{c}_l \in \textup{Port}_{\textup{II}}(\boldsymbol{\mathsf{c}})$.  By Lemma \ref{lem: port II cond}, either
$\F^n (\mathsf{c}_l) $ is critical for some $n < N^1_{\boldsymbol{\mathsf{c}}}(\mathsf{c}_l)$ or
$\Ret_{\boldsymbol{\mathsf{c}}}(\mathsf{c}_l)$ has two distinct children in $\boldsymbol{\mathsf{c}}$.  Neither of these
conditions can hold since $\crit(\tree) = \boldsymbol{\mathsf{c}}$ and every vertex of $\boldsymbol{\mathsf{c}}$ has exactly one
child that is in $\boldsymbol{\mathsf{c}}$.

\end{proof}
\end{lem}

Finally we classify type~III $\boldsymbol{\mathsf{c}}$-portals.

\begin{cor} \label{cor: catagorize bi-critical III portals}
If $(\tree, \F)$ is a bi-critical tree with dynamics with critical end $\boldsymbol{\mathsf{c}} = (\mathsf{c}_l)_{l \in \bZ}$,
then $\textup{Port}_{\textup{III}}(\boldsymbol{\mathsf{c}}) = \set{\mathsf{c}_0}$.
\begin{proof}
By Lemma \ref{lem: Port III categorized}, $\textup{Port}_{\textup{III}}(\boldsymbol{\mathsf{c}}) = \set{\mathsf{c}_l: \
\deg\mathsf{c}_l > \max \mathsf{c}_l^{\child}} $. The only vertex of $\boldsymbol{\mathsf{c}}$ where the degree drops is
$\mathsf{c}_0$.
\end{proof}
\end{cor}

We obtain the following corollary of Theorem \ref{thm: N1 cond} for bi-critical trees with dynamics.

\begin{cor} \label{cor: N1 cond for bi-critical}
Let $(\tree, \F)$ be a bi-critical tree with dynamics.  Let $\boldsymbol{\mathsf{c}} = (\mathsf{c}_l)_{l \in \bZ}$ be the
(unique) critical extended end of $\tree$. The following conditions hold for each $\mathsf{v} \in \tree$:

\begin{enumerate}
    \item If $\mathsf{v} \in \tree_l$, then $ N^1_{\mathsf{\boldsymbol{\mathsf{c}}}}(\mathsf{c}_l) =1$ if $l \leq 1$ and $1
        \leq N^1_{\mathsf{\boldsymbol{\mathsf{c}}}}(\mathsf{c}_l) \leq l$ if $l \geq 2$;

    \item For any child of $\mathsf{v}$, $N^1_{\mathsf{\boldsymbol{\mathsf{c}}}}(\mathsf{v}^{\child}) =
        N^R_{\mathsf{\boldsymbol{\mathsf{c}}}}(\mathsf{v})$ for some $R \geq 1$;

    \item If $N^1_{\mathsf{\boldsymbol{\mathsf{c}}}}(\mathsf{v}^{\child}) = N^R_{\boldsymbol{\mathsf{c}}}(\mathsf{v})$ for
        some $R \geq 2$, then either
\begin{enumerate}
    \item $\Ret^{R-1}_{\boldsymbol{\mathsf{c}}}(\mathsf{v})= \mathsf{c}_k$ for some $k$ such that
        $N^1_{\boldsymbol{\mathsf{c}}}(\mathsf{c}_{k+1} )> N^1_{\boldsymbol{\mathsf{c}}}(\mathsf{c}_k)$;

    \item $\Ret^{R-1}_{\boldsymbol{\mathsf{c}}}(\mathsf{v})= \mathsf{c}_0$.
\end{enumerate}

\end{enumerate}

\begin{proof}
Assume that Theorem \ref{thm: N1 cond} has been proved. Note that $H=1$ in a bi-critical tree with dynamics.  So Condition 1
follows from Condition 1 of Theorem \ref{thm: N1 cond}.  Condition 2 is the same as in Theorem \ref{thm: N1 cond}.  For Condition
3, we have $\Ret^{R-1}_{\boldsymbol{\mathsf{c}}}(\mathsf{v})$ is a $\boldsymbol{\mathsf{c}}$-portal by the Main Lemma (\ref{main
lem}.  By Lemma \ref{lem: no bi-critical II portals}, it cannot be a type~II $\boldsymbol{\mathsf{c}}$-portal.  If it is type~I,
Condition 3a holds by Lemma \ref{lem: bi-critical type I portals}.  Otherwise it is type~III, so Corollary \ref{cor: catagorize
bi-critical III portals} tells us it can only be $\mathsf{c}_0$.
\end{proof}

\end{cor}

Corollary \ref{cor: bi-crit realization of tau} shows that the above conditions are also sufficient for bi-critical polynomials.

Whether or not $\mathsf{c}_0$ is a simple $\boldsymbol{\mathsf{c}}$-portal has important consequences. Branner and Hubbard gave 3
axioms for tableaux, which they claimed were necessary and sufficient conditions for a tableau to be the tableau of a
uni-persistent polynomial with exactly one escaping critical point \cite[Prop.\ 12.8]{BH92}. D.~Harris later found that these
axioms were not in fact sufficient to realize a tableau as a cubic polynomial, and proposed a fourth tableau axiom \cite{Harris}.
J.~Kiwi found a counter-example that shows Harris' axiom is not sufficient. Kiwi gave another fourth tableau axiom, which he
showed is sufficient for cubic polynomials \cite{Kiwi_Puiseux}. None of these previous works have addressed non-cubic
polynomials.  We ask precisely which polynomials require the fourth axiom?  We give a complete answer to this question in
Proposition \ref{prop: tab axioms}.

The key point of Harris' argument is equivalent to the fact that the tree with dynamics of a cubic polynomial has exactly two
vertices at level 1: one of which is critical and the other of which is non-critical.  Hence his argument generalizes to
bi-critical polynomials with a simple escaping critical point.  By the above lemma, this covers bi-critical trees with dynamics
such that $\mathsf{c}_0$ is a simple $\boldsymbol{\mathsf{c}}$-portal.  Hence we consider the first return map when
$\mathsf{c}_0$ is a simple $\boldsymbol{\mathsf{c}}$-portal, and derive some conditions.

Corollary \ref{cor: c-l simple iff deg c-1 - deg c-l+1 =1} implies that there is a simple relationship between the degree of the
escaping critical point of a bi-critical polynomial and whether $\mathsf{c}_0$ is simple or compound.

\begin{lem} \label{lem: bi-crit c0 simple iff simpe esc}
If $(\tree, \F)$ is a bi-critical tree with dynamics with critical end $\boldsymbol{\mathsf{c}} = (\mathsf{c}_l)_{l \in \bZ}$,
then $\mathsf{c}_0$ is a simple $\boldsymbol{\mathsf{c}}$-portal if and only if $\deg \mathsf{c}_0 = \deg \mathsf{c}_1 +1 $.

\end{lem}

\begin{lem} \label{lem: Tab d}
Let  $(\tree, \F)$ be a bi-critical tree with dynamics with critical end $\boldsymbol{\mathsf{c}} = (\mathsf{c}_l)_{l \in \bZ}$.
Suppose that $\mathsf{c}_0$ is a simple $\boldsymbol{\mathsf{c}}$-portal, and $\Ret_{\boldsymbol{\mathsf{c}}} (\mathsf{c}_l) =
\Ret_{\boldsymbol{\mathsf{c}}} (\mathsf{c}_{l+1}) = \mathsf{c}_0$ for some $l$. If $\mathsf{c}_l^{\child} \neq \mathsf{c}_{l+1}$
is a child of $\mathsf{c}_l$, then $\Ret_{\boldsymbol{\mathsf{c}}} (\mathsf{c}_l^{\child}) = \mathsf{c}_{1}$.

\begin{proof}
Since $\mathsf{c}_0$ is a simple $\boldsymbol{\mathsf{c}}$-portal, there are exactly two distinct children of $\mathsf{c}_0$, say
$\set{\mathsf{c}_1, \mathsf{c}_0^{\child_0}}$ by Lemma \ref{lem: bi-crit c0 simple iff simpe esc}. We apply Lemma \ref{lem: F^N a
D-fold cover} and note that $\F^N:\set{\mathsf{c}_l^{\child}} \to \set{\mathsf{c}_0^{\child}}$ is a $(\deg \mathsf{c}_l)$-fold
cover where $N = N^1_{\boldsymbol{\mathsf{c}}}(\mathsf{c}_{l})$. By assumption $\Ret_{\boldsymbol{\mathsf{c}}} (\mathsf{c}_{l+1})
= \mathsf{c}_0 \neq \mathsf{c}_1$, so $\F^N(\mathsf{c}_{l+1}) = \mathsf{c}_0^{\child_0}$.  Since $\deg \mathsf{c}_{l+1} = \deg
\mathsf{c}_{l}$, $\mathsf{c}_{l+1}$ accounts for all of the $\deg \mathsf{c}_{l}$ children of  $\mathsf{c}_{l}$ that are mapped
to $ \mathsf{c}_0^{\child_0}$. Hence for any other child of $\mathsf{c}_{l}$, $\F^N(\mathsf{c}_{l}^{\child})  = \mathsf{c}_{1}$.
Since $N^1_{\boldsymbol{\mathsf{c}}}(\mathsf{c}_{l}^{\child}) \geq N^1_{\boldsymbol{\mathsf{c}}}(\mathsf{c}_{l}) = N$, we have
$N^1_{\boldsymbol{\mathsf{c}}}(\mathsf{c}_{l}^{\child}) = N$ and $\Ret_{\boldsymbol{\mathsf{c}}} (\mathsf{c}_l^{\child}) =
\mathsf{c}_{1}$.

\end{proof}

\end{lem}


The following corollary is a translation of the fourth tableaux axiom of Kiwi \cite{Kiwi_Puiseux} into the language of a tree
with dynamics.

\begin{cor} \label{cor: Tab d}
Under the hypotheses of Lemma \ref{lem: Tab d}, if $\Ret^S_{\boldsymbol{\mathsf{c}}} (\mathsf{v}) = \mathsf{c}_l$ and
$N^1_{\boldsymbol{\mathsf{c}}} (\mathsf{v}^{\child}) > N^S_{\boldsymbol{\mathsf{c}}} (\mathsf{v})$ for some $\mathsf{v} \in \tree
$ and some $S \geq 1$, then $\Ret_{\boldsymbol{\mathsf{c}}} (\mathsf{v}^{\child}) =  \mathsf{c}_{1}$.
\end{cor}

In the above corollary although $\mathsf{c}_{0}$ is a $\boldsymbol{\mathsf{c}}$-portal, a vertex cannot pass through
$\mathsf{c}_{0}$ after hitting $\mathsf{c}_{l}$.  Hence we say that $\mathsf{c}_{l}$ \emph{obstructs} $\mathsf{c}_{0}$.

We obtain the following result for tableaux, which clarifies a result of Branner and Hubbard \cite[Prop.\ 12.8]{BH92} and
generalizes a result of Kiwi \cite{Kiwi_Puiseux}. All 4 tableaux axioms for marked grids (Ma--Md) are given in
\cite{Kiwi_Puiseux}.

\begin{prop} \label{prop: tab axioms}
If $f$ is a bi-critical polynomial with a disconnected Julia set, then any tableau of $f$ satisfies Ma--Mc.  Moreover if the
escaping critical point of $f$ has multiplicity 1, then any tableau of $f$ also satisfies Md. Conversely a tableau that satisfies
Ma--Mc can be realized by a bi-critical polynomial whose escaping critical point has multiplicity at least 2.  If the tableau
also satisfies Md, then it can be realized by bi-critical polynomial whose escaping critical has multiplicity at least 1.

\begin{proof}
The hard work for this proposition was already done in \cite[Prop.\ 12.8]{BH92} and \cite{Kiwi_Puiseux}. The necessary part
follows from and Corollary \ref{cor: Tab d}. The conditions for sufficiency follow from Corollary \ref{cor: bi-crit realization
of tau}.

\end{proof}

\end{prop}

\begin{rem}
In fact we can weaken the assumption that $f$ is bi-critical. We only need to assume that the tree with dynamics of $f$ is
bi-critical.  It is straightforward to check from the definitions in \S\ref{subsect: Poly TwD} that the tree with dynamics of $f$
is bi-critical if and only if all persistent critical points of $f$ lie in the same connected component of $\Kjulia_f$ and the
Green's function of $f$ is constant on the set of escaping critical points of $f$.
\end{rem}

\subsection{Uni-Persistent Polynomials} \label{subsect: uni-persistent returns}

In this subsection, we consider a general uni-persistent polynomial.  That is, a polynomial that has exactly one persistent
critical point and possibly multiple escaping critical points.  The corresponding tree with dynamics has a unique critical end
$\boldsymbol{\mathsf{c}}$, but $\boldsymbol{\mathsf{c}} \varsubsetneq \crit(\tree)$ in general.  Thus
$\Ret_{\boldsymbol{\mathsf{c}}}$ and $\Ret_{\crit(\tree)}$ are generally different maps, so we must choose which one to use. As
above we consider $N^1_{\boldsymbol{\mathsf{c}}}$ and $\Ret_{\boldsymbol{\mathsf{c}}}$. Although the return map to the critical
set is an equally natural choice and some results about it are known \cite{E03}.

We classify $\boldsymbol{\mathsf{c}}$-portals.  Lemmas \ref{lem: catagorize uni-per I portals} and \ref{lem: Port III
categorized} apply to uni-persistent tree with dynamics, so we have already classified type~I and type~III
$\boldsymbol{\mathsf{c}}$-portals.

\begin{lem}\label{lem: 1CB => P2 finite}
Let $(\tree, \F)$ be a uni-persistent tree with dynamics. If $\boldsymbol{\mathsf{c}}$ is the unique critical end of $(\tree,
\F)$, then $\textup{Port}_{II}(\boldsymbol{\mathsf{c}})$ is finite.

\begin{proof}
By Lemma \ref{lem: port II cond} and Corollary \ref{cor: Type II portals of an end},
$\textup{Port}_{\textup{II}}(\boldsymbol{\mathsf{c}})$ is the set of all vertices $\mathsf{c}_L$ that hit a critical vertex not
in $\boldsymbol{\mathsf{c}}$ before returning to $\boldsymbol{\mathsf{c}}$. It is straightforward to check that if
$\boldsymbol{\mathsf{c}}$ is the unique critical end of $(\tree, \F)$, then $\tree$ has only finitely many critical vertices that
are not in $\boldsymbol{\mathsf{c}}$.
\end{proof}
\end{lem}

We have classified all $\boldsymbol{\mathsf{c}}$-portals in the uni-persistent case.  We are now ready to prove our main theorem
for return times.

\begin{proof}[Proof of Theorem \ref{thm: N1 cond}]

There are two cases in Condition~1. If $l \leq H$, apply Lemma \ref{lem: N_1 T_l = 1 for l <H}. If $l > H$, say $l = kH + h$ for
some $k \in \bZ^+$ and $0 < h \leq H$. Then $\F^{k+1}(\mathsf{v}) =\mathsf{v}_{h-H} \in\boldsymbol{\mathsf{c}} $ since $h -H \leq
0$ by T4. So $N^1_{\boldsymbol{\mathsf{c}}}(\mathsf{v}) \leq k+1 = \lceil l/H \rceil$.  Condition~2 is a consequence of Lemma
\ref{lem: ret of descen => ret}. Condition~3 follows from our Main Lemma (\ref{main lem}) and the classification of
$\boldsymbol{\mathsf{c}}$-portals.
\end{proof}

The key assumption in Lemma \ref{lem: Tab d} is that $\mathsf{c}_0$ is a simple $\boldsymbol{\mathsf{c}}$-portal.    We show that
the any simple $\boldsymbol{\mathsf{c}}$-portal can be obstructed.

\begin{lem} \label{lem: Returns to simple portals}
Let  $(\tree, \F)$ be a uni-persistent tree with dynamics with critical end $\boldsymbol{\mathsf{c}} = (\mathsf{c}_l)_{l \in
\bZ}$. Let $\mathsf{c}_m$ be a simple $\boldsymbol{\mathsf{c}}$-portal for some $m$. Suppose that for some $l$,
$\Ret_{\boldsymbol{\mathsf{c}}} (\mathsf{c}_l) =  \mathsf{c}_m$, $N^1_{\boldsymbol{\mathsf{c}}} (\mathsf{c}_{l+1}) =
N^2_{\boldsymbol{\mathsf{c}}} (\mathsf{c}_{l})$, $\deg \mathsf{c}_l = \deg \mathsf{c}_{l+1}$, and $\F^n(\mathsf{c}_l)$ is
non-critical for $1 \leq n < N^1_{\boldsymbol{\mathsf{c}}}(\mathsf{c}_l)$. If $\mathsf{c}_{l}^{\child} \neq \mathsf{c}_{l+1}$ is
a child of $\mathsf{c}_{l}$, then $N^1_{\boldsymbol{\mathsf{c}}} (\mathsf{c}_l^{\child}) \neq N^1_{\boldsymbol{\mathsf{c}}}
(\mathsf{c}_{l+1})$.

\begin{proof}
Since $\mathsf{c}_{m} $ is a simple $\boldsymbol{\mathsf{c}}$-portal, it has unique child of $\mathsf{c}_{m}^{\child_0} $ such
that $\mathsf{c}_{m}^{\child_0} \notin \boldsymbol{\mathsf{c}}$ and $N^1_{\boldsymbol{\mathsf{c}}}(\mathsf{c}_m) =
N^1_{\boldsymbol{\mathsf{c}}}(\mathsf{c}_{m}^{\child_0})$. Let $N = N^1_{\boldsymbol{\mathsf{c}}}(\mathsf{c}_{l})$.  Since
$\F^n(\mathsf{c}_l)$ is non-critical for $1 \leq n < N$, $\F^N:\set{\mathsf{c}_l^{\child}} \to \set{\mathsf{c}_m^{\child}}$ is a
$(\deg \mathsf{c}_l)$-fold cover  by Lemma \ref{lem: F^N a D-fold cover}. Since $\deg \mathsf{c}_l = \deg \mathsf{c}_{l+1}$,
$\mathsf{c}_{l+1}$ accounts for all the children of $\mathsf{c}_{l}$ that are mapped to $\mathsf{c}_{m}^{\child_0}$. Hence if
$\mathsf{c}_{l}^{\child} \neq \mathsf{c}_{l+1}$, then $F^N(\mathsf{c}_{l}^{\child}) \neq \mathsf{c}_{m}^{\child_0}$. Thus
$N_1(\mathsf{c}_{l}^{\child}) \neq N + N^1_{\boldsymbol{\mathsf{c}}}(\mathsf{c}_m)$, and $ N_1(\mathsf{c}_{l}^{\child}) \neq
N^2_{\boldsymbol{\mathsf{c}}}(\mathsf{c}_l) = N^1_{\boldsymbol{\mathsf{c}}} (\mathsf{c}_{l+1})$.
\end{proof}

\end{lem}


In a manner similar to Corollary \ref{cor: Tab d}, $\mathsf{c}_{l}$ obstructs $\mathsf{c}_{m}$.

\begin{cor}
Under the hypotheses of Lemma \ref{lem: Returns to simple portals}, if $\Ret^S_{\boldsymbol{\mathsf{c}}} (\mathsf{v}) =
\mathsf{c}_l$ and $N^1_{\boldsymbol{\mathsf{c}}}(\mathsf{v}^{\child}) > N^S_{\boldsymbol{\mathsf{c}}}(\mathsf{v})$ for some
$\mathsf{v} \in \tree$ and $S \geq 1$, then $N^1_{\boldsymbol{\mathsf{c}}}(\mathsf{v}^{\child}) \neq
N^{S+2}_{\boldsymbol{\mathsf{c}}}(\mathsf{v})$. Thus $\Ret_{\boldsymbol{\mathsf{c}}} (\mathsf{v}^{\child}) \neq
\Ret_{\boldsymbol{\mathsf{c}}} (\mathsf{c}_{l+1})$.
\begin{proof}
We have $F^{N^S_{\boldsymbol{\mathsf{c}}}(\mathsf{v})}(\mathsf{v}^{\child}) = \mathsf{c}_{l}^{\child} \neq \mathsf{c}_{l+1} $
since $N^1_{\boldsymbol{\mathsf{c}}}(\mathsf{v}^{\child}) > N^S_{\boldsymbol{\mathsf{c}}}(\mathsf{v})$. Thus
$\Ret_{\boldsymbol{\mathsf{c}}}(\mathsf{v}^{\child}) = \Ret_{\boldsymbol{\mathsf{c}}}(\mathsf{c}_{l}^{\child})$ and
$\Ret_{\boldsymbol{\mathsf{c}}}(\mathsf{c}_{l}^{\child}) \neq \Ret_{\boldsymbol{\mathsf{c}}} (\mathsf{c}_{l+1})$ by Lemma
\ref{lem: Returns to simple portals} since their return times are different.
\end{proof}

\end{cor}

\section{Yoccoz Return Functions} \label{sect: Yoccoz Ret}

One method of describing the returns of an end to itself is what is generally called the ``Yoccoz $\tau$-function.''  Yoccoz
defined $\tau$ for quadratic polynomials, and his definition easily extends to bi-critical polynomials with disconnected Julia
sets \cite[Rem.\ 9.3]{Hubbard_Loc_Con}. There are several ways to generalize this function to more general polynomials. We wish
to emphasize the particular generalization which we work with. So we use more descriptive terminology-- \emph{the Yoccoz return
function}.

In this section, we define the Yoccoz return function. We extend our results from \S\ref{sect: portals} to these functions
(\S\ref{subsect: properties of tau}). This allows us to prove Theorem \ref{thm: tau cond}, which implies Theorem \ref{main thm}.
This theorem provides enough information to recursively construct a Yoccoz return function (\S\ref{subsect: recursive defs of
tau}). However such a construction is not a simple matter of choosing some initial values and using a recursive relation to
define the function. We give some results on legal ways to extend a Yoccoz return function. We show that constructing such a
function requires a choice at infinitely many stages of the construction (Proposition \ref{prop: 2 choices to extend tau}). We
also give some examples illustrate the process.

\begin{defn} \label{defn: tau}
Let $\boldsymbol{\mathsf{c}} = (\mathsf{c}_l)_{l \in \bZ}$ be a critical extended end of a tree with dynamics. We define the
\emph{Yoccoz return function of} $\boldsymbol{\mathsf{c}}$  as the map $\tau: \bZ \to \bZ$ such that
\[
    \tau(l)  = l - N^1_{\boldsymbol{\mathsf{c}}}(\mathsf{c}_l)  H .
\]
Equivalently, $\Ret_{\boldsymbol{\mathsf{c}}}(\mathsf{c}_l) = \mathsf{c}_{\tau(l)}$.
\end{defn}

\begin{rem}
The above definition generalizes the previous definition of $\tau$ \cite[Rem.\ 9.3]{Hubbard_Loc_Con} in 3 ways. The first 2
differences are minor generalizations.  We define $\tau:\bZ \to \bZ$ instead of $\tau:\bZ^+ \to \bN$, and we allow $H >1$.  The
last difference is substantive. Previously $\tau$ was defined in terms of a tableau by $\tau(l) = l-n$ where $n \geq 1$ is least
such that the tableau position $(l-n,n)$ is critical. In this definition, $n$ is first return time of $\mathsf{c}_l$ to the
critical set, but not necessarily the first return time to critical end $\boldsymbol{\mathsf{c}}$. We use the above
generalization because it induces a dynamically meaningful map on the integers. The two definitions agree for bi-critical
polynomials. The differences come from how one generalizes $\tau$ for polynomials with more than one escaping critical point.
\end{rem}

The Yoccoz return function a persistent critical point of a polynomial is the Yoccoz return function of the critical end of the
tree with dynamics of the polynomial associated with the nest of the persistent critical point. Informally we refer to Yoccoz
return functions as ``$\tau$-functions.''

Suppose $\tau$ is the Yoccoz return function of the critical extend end of a uni-persistent polynomial.  Then $\tau $ function
codes basic dynamical behavior of a persistent critical point. We say $\tau$ is \emph{recurrent} if $\limsup_{l \to \infty}
\tau(l) = \infty$.  We say $\tau$ is \emph{persistently recurrent} if $\liminf_{l \to \infty} \tau(l) = \infty$.   It is well
known that $\boldsymbol{\mathsf{c}}$ is recurrent if and only if $\tau$ is recurrent. Persistent recurrence is a combinatorial
condition. It is well known that if $\tau$ is recurrent, but not persistently recurrent, then the Julia set of the polynomial is
an area zero Cantor set.

\subsection{Properties of Yoccoz Return Functions}\label{subsect: properties of tau}

We derive some properties of Yoccoz return functions.  Our goal is to find a collection conditions that are necessary and
sufficient for realization.

\begin{lem} \label{lem: tau(l+1) = tau^R(l)}
For any $l \in \bZ$, $\tau(l) = \tau^R(l-1) + 1$ for some $R \geq 1$.
\begin{proof}
Combine Lemma \ref{lem: ret of descen => ret} and Definition \ref{defn: tau}.
\end{proof}

\end{lem}

\begin{defn}
Let $\tau: \bZ \to \bZ$ be a Yoccoz return function.  For each $l \in \bZ$, we define $R(l)\in \bZ^+$ so that
\[
    \tau(l) = \tau^{R(l)}(l-1) + 1.
\]
\end{defn}

It follows from Lemma \ref{lem: tau(l+1) = tau^R(l)} that $R(l)$ is well defined. Note that $\tau(l) = \tau(l-1)+1$ if and only
$R(l) =1$.

We restate Theorem \ref{thm: N1 cond} in terms of the $\tau$-function, which expresses Theorem \ref{main thm} in terms of a tree
with dynamics.

\begin{thm} \label{thm: tau cond}
If $\tau$ is the Yoccoz return function of a critical end $\boldsymbol{\mathsf{c}}$ of a tree with dynamics, then for some $H \in
\bZ^+$ and some $E \subset \bN$ with $0 \in E$ the following conditions hold for each $l \in  \bZ$:

\begin{enumerate}
    \item $\tau(l)= l-H$ if $l \leq H$, and $- H < \tau (l) < l $ if $l >H$;
    \item $\tau(l) = \tau^R(l-1)+1$ for some $R = R(l)  \geq 1$;
    \item if $R(l)  \geq 2$, then either
        \begin{enumerate}
            \item $\tau(\tau^{R-1}(l-1) + 1) \leq \tau^R(l-1)$,
            \item $\tau^{R-1}(l-1) \in E$.
        \end{enumerate}
\end{enumerate}
Moreover if $\boldsymbol{\mathsf{c}}$ is the only critical end of the tree with dynamics, then $E$ is finite.

\end{thm}

If $E$ is infinite, then condition 3 is fairly weak. Thus to study a tree with dynamics with several critical ends (equivalently
a polynomial with several persistent critical points) we need to modify our combinatorics.

We start proving the results we use to prove the above theorem.

\begin{lem}  \label{lem: tau(l) = l-H for l leq H}
If $l \leq H$, then $\tau(l) = l-H$.

\begin{proof}
Apply Lemma \ref{lem: N_1 T_l = 1 for l <H}.
\end{proof}

\end{lem}

As a result of Lemmas \ref{lem: tau(l+1) = tau^R(l)} and \ref{lem: tau(l) = l-H for l leq H}, we can recover $\tau$ from $R$.
Hence $\tau$ and $R$ are equivalent combinatorial objects.

We extend the concept of portals to $\tau$-functions.

\begin{defn}
We call $l \in \bZ$ a $\tau$-\emph{portal} if  $\tau(l+1) \leq \tau(l)$.
\end{defn}

\begin{lem} \label{lem: tau-portal equivalences}

For any $l \in \bZ$, the following are equivalent:

\begin{enumerate}
  \item $l $ is a $\tau$-{portal},
  \item $R(l+1) \geq 2$,
  \item $N^1_{\boldsymbol{\mathsf{c}}}(\mathsf{c}_{l}) < N^1_{\boldsymbol{\mathsf{c}}}(\mathsf{c}_{l+1})$.
\end{enumerate}

\begin{proof}
We show that $1 \implies 2 \iff 3$. If $\tau(l+1) \leq \tau(l)$, then $\tau(l+1) \neq \tau^1(l)+1$. So $R(l+1) \geq 2$ and $1
\implies 2$ is shown. We know $N^R_{\boldsymbol{\mathsf{c}}}(\mathsf{c}_{l}) = N^1_{\boldsymbol{\mathsf{c}}}(\mathsf{c}_{l+1})$
for some $R \geq 1$ from Lemma \ref{lem: ret of descen => ret}.  Since the return times of $\mathsf{c}_{l}$ are distinct,
$N^1_{\boldsymbol{\mathsf{c}}}(\mathsf{c}_{l}) = N^1_{\boldsymbol{\mathsf{c}}}(\mathsf{c}_{l+1})$ if and only if $R(l+1) = 1$.
Which is to say not $3 \iff \text{ not }2$.
\end{proof}

\end{lem}

We consider the set where Condition 3b holds.

\begin{defn} \label{defn: E for tau}
Let $\tau$ be the Yoccoz return function of a critical end $\boldsymbol{\mathsf{c}}$. We define the \emph{exceptional set of}
$\tau$ by $E  =\set{l \in \bZ: \ \mathsf{c}_l \text{ is a $\boldsymbol{\mathsf{c}}$-portal, but $l$ is not a $\tau$-portal}}$.
\end{defn}

\begin{lem}
Let $(\tree,\F)$ be a uni-persistent tree with dynamics.  If $\tau$ is the Yoccoz return function of the unique critical end of
$(\tree,\F)$, then $E$ is finite.
\begin{proof}
Note that if $\mathsf{c}_l \in \textup{Port}_{\textup{I}}(\boldsymbol{\mathsf{c}})$, then $l$ is a $\tau$-portal. Thus
\[
    E \subset \set{l: \ \mathsf{c}_l \in \textup{Port}_{\textup{II}}(\boldsymbol{\mathsf{c}}) \cup
    \textup{Port}_{\textup{III}}(\boldsymbol{\mathsf{c}})}.
\]
Apply Lemma \ref{lem: 1CB => P2 finite} and Corollary \ref{cor: finite number of type 3 portals}.
\end{proof}

\end{lem}

\begin{lem} \label{lem: 0 in E}
For the Yoccoz return function of any critical end of a tree with dynamics, $0 \in E$.
\begin{proof}
By Lemma \ref{lem: v_0 a type 3 portal}, $\mathsf{v}_0 =  \mathsf{c}_0$ is a type~III $\boldsymbol{\mathsf{c}}$-portal. Since
$N^1_{\boldsymbol{\mathsf{c}}}(\mathsf{c}_{1})  = N^1_{\boldsymbol{\mathsf{c}}}(\mathsf{c}_{0}) = 1$ by Lemma \ref{lem: N_1 T_l =
1 for l <H}, $0$ is not a $\tau$-portal. Therefore $0 \in E$.

\end{proof}

\end{lem}

We now prove Theorem \ref{thm: tau cond} which implies Theorem \ref{main thm}.

\begin{proof}[Proof of Theorem \ref{thm: tau cond}.]
We have $0\in E$ from Lemma \ref{lem: 0 in E}. The first part of Condition 1 is Lemma \ref{lem: tau(l) = l-H for l leq H}.  It
follows by a routine induction argument from Lemmas \ref{lem: tau(l) = l-H for l leq H} and \ref{lem: tau(l+1) = tau^R(l)} that
$\tau(l) < l$ for any $l \in \bZ$. Condition 2 is Lemma \ref{lem: tau(l+1) = tau^R(l)}. Condition 3 is a consequence of our Main
Lemma. Say $\Ret^{R(l)-1}_{\boldsymbol{\mathsf{c}}}(\mathsf{c}_{l-1}) =\mathsf{c}_{k}$. Then $\mathsf{c}_{k}$ is a
$\boldsymbol{\mathsf{c}}$-portal. If $N^1_{\boldsymbol{\mathsf{c}}}(\mathsf{c}_{k}) <
N^1_{\boldsymbol{\mathsf{c}}}(\mathsf{c}_{k+1})$, then $k$ is a $\tau$-portal by Lemma \ref{lem: tau-portal equivalences}.
Otherwise $k \in E$ by Definition \ref{defn: E for tau}.   Moreover in uni-persistent tree with dynamics,
$\textup{Port}_{II}(\boldsymbol{\mathsf{c}})$ and $\textup{Port}_{III}(\boldsymbol{\mathsf{c}})$ are finite by Lemma \ref{lem:
1CB => P2 finite} and Corollary \ref{cor: finite number of type 3 portals} respectively. Thus $E \subset
\textup{Port}_{II}(\boldsymbol{\mathsf{c}}) \cup \textup{Port}_{III}(\boldsymbol{\mathsf{c}})$ is also finite.
\end{proof}

The conditions in Theorem \ref{thm: tau cond} are also sufficient. We (temporarily) call a function $\tau: \bZ \to \bZ$ that
satisfies Conditions 1--3 above for some $H$ and some finite set $E \subset \bN$ with $0 \in E$ an \emph{abstract Yoccoz return
function}. Theorem \ref{thm: realize tau} shows that every abstract Yoccoz return function is the Yoccoz return function of the
critical end of a uni-persistent tree with dynamics.

Finally we translate Lemma \ref{lem: Returns to simple portals} into the language of $\tau$-functions.

\begin{lem} \label{lem: tau simple portals}
Let $m \in E$ such that $\mathsf{c}_{m}$ is a simple $\boldsymbol{\mathsf{c}}$-portal. Suppose that $\tau(l) = m $ and $R(l+1) =
2$ for some $l$. If $\tau^S(k) = l$ and $R(k+1) >S$ for some $k$ and $S\geq 1$, then $\tau(k+1) \neq \tau(l+1)$. Thus $R(k+1)
\neq S+2$.

\end{lem}

For $k,l,m$ in the above lemma, we say that $l$ obstructs $m$.\\

\subsection{Recursively Defining Yoccoz Return Functions}\label{subsect: recursive defs of tau}

Theorem \ref{thm: tau cond} provides recursively verifiable conditions that must be satisfied by $\tau$-functions.  Since these
conditions are quite complicated, we give some examples of how they can be used to recursively define a $\tau$-function. For
simplicity we will we assume $H=1$ and $E = \set{0}$ in our examples. So we are considering the case of the $\tau$-function of a
bi-critical tree with dynamics.

\begin{defn}
A \emph{Yoccoz return function of length} $L \in \bZ^+$ is a function
\[
    \tau: \set{l \in \bZ: l \leq L} \to \bZ
\]
which satisfies the 3 conditions of Theorem \ref{thm: tau cond}.

\end{defn}

In order to define an abstract Yoccoz return function (with $H=1$), we must start by defining $\tau(l) = l-1$ for $l \leq 1$ by
Condition 1. Equivalently, $R(l)= 1$ for $l \leq 1$. The problem is to take a Yoccoz return function of length $L$ and extend it
to length $L+1$. The key question is what are the valid choices for $R(L+1)$?

\begin{example} \label{ex: tau of 2}

We have a choice for $\tau(2)$.  We can choose $R(2) = 1$ since it satisfies Condition 2 and vacuously satisfies Condition 3.  If
$R(2) =1$, then $\tau(2) = \tau(1) + 1 = 1$. Also $\tau(1) = 0 \in E$, so we can choose $R(2) = 2$ by Condition 3b. In this case,
$\tau(2) = \tau^2(1) + 1 = 0$ and 1 is a $\tau$-portal.

\end{example}

We generalize the arguments in the above example, and give a few results on extending a Yoccoz return function of length $L$. The
simplest choice is to take $R(L+1) = 1$.

\begin{lem} \label{lem: can choose R(L+1) = 1}
Given $\tau$ a Yoccoz return function of length $L$, the choice $R(L+1) =1$ extends $\tau$ to length $L+1$.
\begin{proof}
The choice $R(L+1) =1$ satisfies Condition 1 since $\tau(L)$ does. It clearly satisfies Condition 2 and vacuously satisfies
Condition 3.
\end{proof}

\end{lem}

In Example \ref{ex: tau of 2} there are 2 valid choices for $R(2)$.  We isolated the property that allowed us to make the choice
$R(2) = 2$.

\begin{defn} \label{defn: M(L)}
Let $\tau$ be the Yoccoz return function (possibly of length $L \geq 1$). We define $M(l)$ by
\[
    M(l) = \min \set{S: \ \tau^{S}(l) \leq 0}
\]
for $1 \leq l \ (\leq L)$.

\end{defn}

An easy inductive argument shows that $M(l)$ is well defined since $\tau(k) < k$ for any $k$. We have $l \equiv 0  \pmod H$ if
and only if $\tau^{M(l)}(l) = 0$.

\begin{cor} \label{cor: bounds on R(l)}
If $\tau$ is a Yoccoz return function (possibly of length $L \geq 2$), then $R(l) \leq  M(l-1) + 1$ for any $1 \leq l \ (\leq
L)$.
\begin{proof}
We have $-H < \tau^{M(l-1)}(l-1) \leq 0$ from the definition of $M(l)$ and Condition 1. So if $R(l) > M(l-1) + 1$, then
$\tau^{R(l)}(l-1) \leq -2 H$. So we would have $\tau(l) = \tau^{R(l)}(l-1)+1 \leq -2H +1 \leq -H$ contrary to Condition 1.
\end{proof}

\end{cor}

The above corollary only depends on Conditions 1 and 2. We call $\set{1,\dots, M(L) + 1}$ the \emph{a priori choices} for
$R(L+1)$.

\begin{lem} \label{lem: l equiv 0 mod H => choose R(L+1) = M+1}
Let $\tau$ be a Yoccoz return function of length $L \geq 1$. If $L \equiv 0 \pmod H$, then the choice $R(L+1) = M(L) + 1$ extends
$\tau$ to length $L+1$.
\begin{proof}
It is straightforward to check that $\tau^{M(L)}(L) = 0 \in E$.  Thus the choice $R(L+1) = M(L) + 1$ satisfies Condition 3b. Also
$\tau(L+1) = \tau^{M(L)+1}(L)+1 = 0 $ satisfies Condition 1.
\end{proof}
\end{lem}

It follows that we can make a choice when one extends a Yoccoz return function of length $L$ whenever $L \equiv 0 \pmod H$.

\begin{prop} \label{prop: 2 choices to extend tau}
Let $\tau$ be a Yoccoz return function of length $L \geq 1$. If $L \equiv 0 \pmod H$, there are at least two distinct ways to
choose $R(L+1)$ that extends $\tau$ to length $L+1$.

\begin{proof}
We can choose $R(L+1) = 1$ or $R(L+1 )= M(L)+1$ by Lemmas \ref{lem: can choose R(L+1) = 1} and \ref{lem: l equiv 0 mod H =>
choose R(L+1) = M+1} respectively.  Since $L \geq 1$, we have $M(L) \geq 1$. Therefore these are distinct choices.
\end{proof}
\end{prop}

\begin{cor}
For $L \geq 1$ and $H \geq 1$, there are at least $2^{\lfloor (L-1)/H \rfloor}$ Yoccoz return functions of length $L$ with
$\tau(0) = -H$.
\end{cor}

Some of the a priori choices for $R(L+1)$ may not be permitted. The valid choice for $R(L+1)$ are only those $R$ that satisfy
Condition 3, so they depend on the definition of $\tau(2), \dots, \tau(L)$. For instance the valid choices for $\tau(3)$ depend
on how we define $\tau(2)$.

\begin{example} \label{ex: tau of 3}
First suppose that $\tau(2) = 0$.  We can choose $R(3) =1$ and $\tau(3) = 1$.  Since $\tau(2) = 0 \in E$, we can choose $R(3) =2$
and $\tau(3) = 0$. These are the only valid choices by Condition 1. Now suppose that $\tau(2) = 1$. We can choose $R(3) = 1$, so
$\tau(3) =2$.  Since $\tau^2(2) = 0 \in E$, we can choose $R(3) = 3$ and $\tau(3) = 0$.  It is not immediately clear if we can
choose $R(3) = 2$.  We must check Condition 3a, so we ask if $\tau^{2-1}(3-1)$ is a $\tau$-portal?  That is, we check if $\tau(2)
=1$ is a $\tau$-portal. But $R(2) = 1$, so 1 is not a $\tau$-portal by Lemma \ref{lem: tau-portal equivalences}.  Hence if
$\tau(2) = 1$, then $R(3) \neq 2$.

\end{example}

The last case illustrates a key consequence of Condition 3.  The choices of $R(L)$ are restricted. A priori if $\tau(2) = 1$,
then we have that $R(3) \in \set{1,2, 3}$. However we have just shown that $R(3) = 2$ is not a valid choice. In general we should
not expect that all the a priori choices are valid.

We continue Example \ref{ex: tau of 3}. It would be too cumbersome to consider all possible $\tau$-functions of length $L$ with
$H=1$ since there are at least $2^{L-1}$ of them. So we will we consider particular choices for $\tau(2)$ and $\tau(3)$.

\begin{example}
Suppose that $\tau(2) = 0$ and $\tau(3) =1$.  We define $\tau(4)$. We can choose $R(4)= 1$ ($\tau(4) = 2$). Since $\tau^2(3) = 0
\in E$ we can choose $R(4) = 2+1 = 3$ ($\tau(4) = 0$).  We use Condition 3a to check if we can choose $R(4) = 2$:
\[
    \tau(\tau^{2-1}(4-1)+1) = \tau(\tau(3) +1) = \tau(2) =0 \leq 0 = \tau(1) =  \tau^{2}(3).
\]
So $R(4) = 2$ is a possible choice. If we choose $R(4) = 2$ ($\tau(4) = 1$), then 3 is $\tau$-portal.

\end{example}

In the above example all three of the a priori choices are possible.

\begin{example}
Continuing the above example, the cases when $ L = 4$ or $5$ are similar, so suppose we define $R(5) = R(6)= 1$, so $\tau(5) = 2$
and $\tau(6) = 3$. We have defined the following Yoccoz return function of length 6:

\[
\begin{array}{|c|c|c|c|c|c|c|c|c|c|c|c|}\hline
  l         & 1  &  2 &  3 & 4  & 5  &  6 \\
  \hline
  R(l)      &  1 &  2 &  1 &  2 &  1 & 1  \\
  \hline
  \tau(l)   &0  &  0 & 1  & 1  & 2  &  3\\
  \hline
\end{array}
\]

Now we consider $L = 6$ and the possible choices for $R(7)$.  Consider the orbit of $L = 6$: we have $6 \mapsto 3 \mapsto 1
\mapsto 0$. Hence $R(7) \in \set{1,2, 3, 4}$. Of course $R(7)= 1$ ($\tau(7) =4$) is a valid choice. Since $\tau^3(6) = 0 \in E$,
we can choose $R(7) = 4$ and $\tau(7) = 0$. Now $3$ and $1$ are $\tau$-portals, so $R(7) = 2$ or $3$ are also legal choices
corresponding to $\tau(7) = 2$ or $\tau(7) = 1$ respectively.

\end{example}

In the above example, notice how we found valid choices for $R(L+1)$ by tracking the orbit of $L$ and looking for $\tau$-portals.

We give an example of defining a $\tau$-function on all integers, rather than just an initial segment. Branner and Hubbard showed
that there are uni-persistent polynomials where first return times of the persistent critical point are the Fibonacci numbers
\cite[Ex.\ 12.4]{BH92}.  We give a new proof of their result using Theorem \ref{thm: tau cond}. For an abstract Yoccoz return
function, we define the \emph{first return time} of $l \in \bZ$ by $N^1_{\tau}(l) = (l -\tau(l))/H$ (compare to Definition
\ref{defn: tau}).

\begin{example}
Let $(a_k)_{k \in \bN}$ denote the Fibonacci numbers: $a_0 = a_1 =1$  and $a_k = a_{k-1} + a_{k-2}$ for $k \geq 2$. Define
$(b_k)_{k \in \bN}$ by $b_0= 0 $ and $b_k = b_{k-1} + a_k$. We construct a $\tau$-function so that $\tau(b_k) = b_{k-1}$ for $k
\geq 1$. It follows that $N^1_{\tau}(b_k) = a_k$.

We have $b_1 = 1$, and $\tau(b_1) = 0 = b_0$ by Condition 1. Now $b_2 = 1+2 = 3$, and we want $\tau(b_2)= b_1 = 1$.  We define
$\tau(2) = 0 $ and $\tau(3) =1$ (see Example \ref{ex: tau of 2}).  Note that $\tau(2)= 2-a_2$ and $\tau(3) = 3 - a_2$.

Inductively assume that we have defined $\tau(l)$ for $l \leq b_K$ for some $K$ such that for any $l \in \bZ^+$, if $b_{k-1} < l
\leq b_k$ for some $k$ with $1 \leq k \leq K$, then $\tau(l) = l- a_k$. In particular, $\tau(b_k ) = b_{k-1}$ for $1 \leq k \leq
K$. It follows that $b_k$ is a $\tau$-portal for any $1 \leq k < K$. By the inductive hypothesis, $\tau(b_K) = b_{K-1} $ and
$b_{K-1}$ is a $\tau$-portal.  Hence $R(1+b_K) = 2$ is a valid choice, which gives $\tau(1+b_K) = 1+\tau^2(b_K) = 1+b_{K-2}$.
With this choice, $b_K$ is a $\tau$-portal. Using the Fibonacci relation and the definition of $b_k$, we find
\[
    \tau(1+b_K)= 1+b_{K-2} = 1 + b_K -a_{K+1}.
\]
We define $R(l) = 1$ for $1+b_K < l \leq b_{K+1}$. It is easy to check that $\tau(l ) = l  -a_{K+1}$ for $b_{K-1} < l \leq
b_{K+1}$.

\end{example}

We choose $R(1+b_K) = 2$ precisely because the Fibonacci numbers are defined by an order 2 recursion. Recall that for $r\in
\bZ^+$, the $r$-bonacci numbers are defined by the order $r$ recursion $a_k = a_{k-1} + a_{k-2} + \dots + a_{k-r} $ for $k \geq
r$ (we use the initial conditions  $a_{r-2} =a_{r-1} = 1$ and $a_k = 0$ for $k \leq r-3$). It is straightforward to generalize
the above example to give a $\tau$-function whose first return times are the $r$-bonacci numbers.  The main difference is that we
choose $R(1+b_K) = r$ for $K$ sufficiently large.

\section{Realization} \label{sect: Realization}

In this section, we prove Theorem \ref{main thm converse}.  That is, we show that any abstract Yoccoz return function with a
finite exceptional set $E$ can be realized by a uni-persistent polynomial (Theorem \ref{thm: realize tau}). There are two main
parts to the proof. In \S\ref{subsect: TwD from tau}, we construct an uni-persistent tree with dynamics that realizes $\tau$ as
the Yoccoz return function of its critical end (Proposition \ref{prop: tau TwD}). Every tree with dynamics is realizable by a
polynomial \cite{DM-MC-pp}, so this construction gives the desired polynomial. We construct the tree with dynamics level by
level.  Having constructed the tree with dynamics up to some level $L$, and realizing $\tau$ to length $L$, Lemma \ref{lem: ext}
(the \emph{Extension Lemma}) tells us we can extend the tree to level $L+1$ in such a way that we realize $\tau(L+1)$. The
remainder of the section (\S\ref{subsect: Proof of ext lem}) is the technical details of the proof the Extension Lemma. The key
point is to keep track of where in the tree there are portals.


In our construction, points in $E$ will be realized as type~III portals using escaping critical points. We need information about
the degree of the escape, which in turn tells us whether the portal is simple or compound (see Lemma \ref{lem: tau simple
portals}).

\begin{defn} \label{defn: esc(m)}
Let $\tau$ be an abstract Yoccoz return function with exceptional set $E$.  For $m \in E$, define $\esc(m)$ by
\begin{enumerate}
    \item $\esc(m)= 0 $ if  $\set{l: \ \tau^{R(l)}(l-1) = m }  = \emptyset$ and $m \geq 1$;
    \item $\esc(m)= 2 $ if there are distinct $k,l \in \bZ^+$ such that $k \notin E$, $\tau(k)=m$, $R(k+1)=2$ and
        $\tau^{R(l+1)-2}(l) = k$;
    \item $\esc(m)= 1 $ otherwise.
\end{enumerate}
Define $\esc(\tau) =  \sum_{m \in E} \esc(m) $.
\end{defn}

When we construct the tree with dynamics, $\esc(m)$ will correspond to the degree of the escape at $\mathsf{c}_m$ (see
Proposition \ref{prop: tau TwD}).

We restate Theorem \ref{main thm converse}.

\begin{thm} \label{thm: realize tau}
Let $\tau$ be an abstract Yoccoz return function for some finite set $E \subset \bN$ with $0 \in E$. For any integers $C \geq
\esc(\tau)$ and $D \geq 2$, there is a uni-persistent polynomial $f$ of degree $C+D$ with a disconnected Julia set such that
\begin{enumerate}
    \item the multiplicity the persistent critical point of $f$ is $D-1$;
    \item the Yoccoz return function of the persistent critical point of $f$ is $\tau$.
\end{enumerate}

\end{thm}

Note that the degree of $f$ above is at least $\esc(\tau)+2$.

\begin{cor} \label{cor: bi-crit realization of tau}
Let $\tau$ be an abstract Yoccoz return function with $E= \set{0}$ and $H=1$. Let $D \in \bN$ with $D \geq 2$.  Then there is a
bi-critical polynomial which realizes $\tau$.

\begin{enumerate}
  \item If there are distinct $k,l \in \bZ^+$ such that $\tau(k)=0$, $\tau(k+1)=0$ and $\tau(l+1) = 0$, then any polynomial
      realizing $\tau$ must have degree at least $D+2$. In particular, the minimal degree of a polynomial realizing $\tau$ is
      4.

  \item Otherwise $\tau$ can be realized by a polynomial of degree $D+1$.  In particular, $\tau$ can be realized by a cubic
      polynomial.
\end{enumerate}

\begin{proof}
Condition 1 above holds if and only if $\esc(0) =2$ by Definition \ref{defn: esc(m)}. The only statement that does not follow
immediately from Theorem \ref{thm: realize tau} is the degree of the polynomial in Case 1 must be at least $D+2$.  If the degree
were $D+1$, then in the tree with dynamics $\mathsf{v}_0=\mathsf{c}_0$ would be a simple $\boldsymbol{\mathsf{c}}$-portal. So
Corollary \ref{cor: Tab d} would apply, which would contradict the existence of $k$ and $l$.

\end{proof}

\end{cor}

In the above corollary, Condition 1 is the negation of Kiwi's fourth tableau axiom.  Hence tableaux which do not satisfy the
fourth axiom are covered by Condition 1 and those that do are covered by Condition 2.

\subsection{A Tree with Dynamics from $\tau$} \label{subsect: TwD from tau}

In order to construct a tree with dynamics from $\tau$, we need a sequence $(D_l)$ which gives the degree of each critical
vertex.

\begin{defn}\label{defn: tau admissible}
Let $\tau$ be an abstract Yoccoz return function for some finite set $E \subset \bN$ with $0 \in E$. We say a sequence of
positive integers $(D_l)_{l \in \bZ}$ is $\tau$ \emph{admissible} if
\begin{enumerate}
    \item $D_l = D_0$ for each $l \leq 0$;
    \item $D_l - D_{l+1} \geq \esc(l)$ for each $l \geq 0$;
    \item $\min D_l \geq 2$.
\end{enumerate}
\end{defn}

It is easy to define a $\tau$-admissible sequence.

\begin{defn} \label{defn: D_l}
Let $\tau$ be an abstract Yoccoz return function for some finite set $E \subset \bN$ with $0 \in E$. Fix $D \in \bN$ with $D \geq
2$. For each $l \in \bZ$, define
\[
    D_l =  D + \sum_{\set{m \in E: \ m \geq l}} \esc(m) .
\]
\end{defn}

The following properties of $(D_l)$ follow from its definition.

\begin{lem}
Let $\tau$ be an abstract Yoccoz return function for some finite set $E \subset \bN$ with $0 \in E$. Fix $D \in \bZ$ with $D \geq
2$. If $(D_l)$ is the sequence defined in Definition \ref{defn: D_l}, then the following conditions hold:
\begin{enumerate}
    \item $D_l = D +  \esc(\tau)$ for each $l \leq 0$;
    \item $D_l - D_{l+1} = \esc(l)$ for each $l \geq 0$;
    \item $D_l =D$ for each $l > \max E$.
\end{enumerate}

In particular, $(D_l)$ is $\tau$ admissible.

\end{lem}

\begin{cor}
The sequence $(D_l)$ from Definition \ref{defn: D_l} with $D=2$ is the minimal $\tau$-admissible sequence in the sense that if
$(D'_l)$ is any $\tau$-admissible sequence, then $D_l \leq D'_l$ for each $l\in \bZ$. In this sequence, $D_0= \esc(\tau)$ +2.
\end{cor}

We can realize a $\tau$-function by a uni-critical tree with dynamics, and a $\tau$-admissible sequence tells us the degree of
each vertex of the critical end.

\begin{prop} \label{prop: tau TwD}
Let $\tau$ be an abstract Yoccoz return function for some finite set $E \subset \bN$ with $0 \in E$. If $(D_l)_{l \in \bZ}$ is
$\tau$ admissible, then there is a tree with dynamics with a unique critical end $\boldsymbol{\mathsf{c}} = (\mathsf{c}_l)_{l \in
\bZ}$ such that the Yoccoz return function of $\boldsymbol{\mathsf{c}}$ is $\tau$ and $\deg \mathsf{c}_l = D_l$ for each $l \in
\bZ$.
\end{prop}

Before we prove Proposition \ref{prop: tau TwD}, we need some technical results. When constructing a tree with dynamics, we
construct a sequence of finite trees with dynamics.
\begin{defn}\label{defn: tree of finite length}
For $L \in \bZ^+$, a tree of \emph{length} $L$ with dynamics is a tree with levels $\set{\tree_l}_{l \leq L}$, which satisfies
all tree axioms (Definition \ref{defn: Tree Axioms} and \ref{defn: TwD}), except that the vertices of $\tree_L$ have no children.
We also refer to such a tree as a tree of \emph{finite length}.
\end{defn}

Most concepts associated an infinite tree with dynamics make sense for a finite tree, and we will use them without comment. One
concept that we need to make explicit is a finite analogue of an end.

\begin{defn}
Let $(\tree, \F)$ be a tree with dynamics (possibly of length $L$).  A \emph{branch} is a set of vertices
$\boldsymbol{\mathsf{x}} = \set{\mathsf{x}_l}_{l = 0}^L$ where $L \in \bZ^+$ and $\mathsf{x}_l \in \tree_l$ and $\mathsf{x}_{l-1}
= \mathsf{x}_l^{\parent}$ for all $0 < l \ (\leq L)$.
\end{defn}

First return maps and related concepts have obvious generalizations to branches.

\begin{defn}
Given two trees with dynamics: $(\tree,\F)$ with degree function $\deg$ of length $L$ and $(\tree',\F')$ with degree function $
\deg '$ of length $L' > L$, we say that $(\tree', \F')$ is an \emph{extension} of $(\tree, \F)$ if $\tree_l = \tree_{l}'$ for $l
= 0, \dots, L$, $\F'|\tree = \F$ and $\deg\mathsf{v} = \deg' \mathsf{v} $ for all $\mathsf{v} \in \tree$.
\end{defn}

We will always construct extensions where we add one level to a tree. Given a tree with dynamics of length $L$ that realizes
$\tau(1), \dots , \tau(L)$, the key point is to extend in such a way that we realize $\tau(L+1)$. The following lemma gives
conditions when we can do so.

\begin{lem}[Extension Lemma] \label{lem: ext}
Let $\tau$ be an abstract Yoccoz return function for some finite set $E \subset \bN$ with $0 \in E$. Let $(D_l)_{l \in \bZ}$ be
$\tau$ admissible.  Let $(\tree,\F)$ be a tree with dynamics of length $L$.  Suppose that there is a critical branch
$\boldsymbol{\mathsf{c}}=\set{\mathsf{c}_0, \dots , \mathsf{c}_L}$ such that $\deg \mathsf{c}_l = D_l$, and
$\Ret_{\boldsymbol{\mathsf{c}}}(\mathsf{c}_l) = \mathsf{c}_{\tau(l)}$ for $0 \leq l \leq L$.  Also suppose that $\deg \mathsf{v}
= 1$ for $ \mathsf{v} \notin \boldsymbol{\mathsf{c}}$. Then there exists $(\tree',\F')$, an extension of $(\tree,\F)$ of length
$L+1$ with $\mathsf{c}_{L+1} \in \tree_{L+1}'$ such that
\begin{enumerate}
    \item $\mathsf{c}_{L+1}$ is a child of $\mathsf{c}_{L}$,
    \item $\deg \mathsf{c}_{L+1} = D_{L+1}$,
    \item $\Ret_{\boldsymbol{\mathsf{c}}}(\mathsf{c}_{L+1}) = \mathsf{c}_{\tau(L+1)} $.
\end{enumerate}

\end{lem}

The proof of Lemma \ref{lem: ext} is long and technical.  Assuming Lemma \ref{lem: ext} is true for the moment, we prove
Proposition \ref{prop: tau TwD}.

\begin{proof}[Proof of Proposition \ref{prop: tau TwD}.]

We will construct the first few levels of the tree explicitly and then start an inductive procedure.  The tree will satisfy the
following conditions at each stage of the construction: There is a critical branch $\boldsymbol{\mathsf{c}} = \set{\mathsf{c}_0,
\dots \mathsf{c}_L}$, $\deg \mathsf{c}_l = D_l$ for $1 \leq l \leq L$, $\Ret_{\boldsymbol{\mathsf{c}}}(\mathsf{c}_l) =
\mathsf{c}_{\tau(l)}$, and $\deg \mathsf{v} = 1$ for all $ \mathsf{v} \in \tree \minus \boldsymbol{\mathsf{c}}$. Also assume that
for each $m \in E$ with $\esc(m) \geq 1$, $\mathsf{c}_{m} \in \textup{Port}_{III}({\boldsymbol{\mathsf{c}}})$ and if $\esc(m) =
2$, then $\mathsf{c}_{m}$ is a compound $\boldsymbol{\mathsf{c}}$-portal.

Define $\tree_l = \set{\mathsf{v}_l = \mathsf{c}_l}$ with $\F(\mathsf{v}_l) = \mathsf{v}_{l-H}$ and $\deg \mathsf{v}_l = D_l$ for
each $ l \leq 0$.

We have $D_0 - D_1 \geq \esc(0) \geq 1$. Let $k = D_0 - D_1$ and define $\tree_1 = \set{\mathsf{c}_1, \mathsf{v}_0^{\child_1},
\dots, \mathsf{v}_0^{\child_{k}}}$ where $\deg \mathsf{c}_1 = D_1$, $\deg \mathsf{v}_0^{\child_i} = 1$ ($1 \leq i \leq k$), and
$F(\mathsf{v}) = \mathsf{v}_{1-H}$ for every $ \mathsf{v} \in \tree_1$.  See figure \ref{fig: top of tree}. Now $\mathsf{c}_0$
has at least one non-critical child, $\mathsf{v}_0^{\child_1}$, such that $N_1(\mathsf{v}_0^{\child_1}) = 1 =
N^1_{\boldsymbol{\mathsf{c}}}(\mathsf{c}_0)$. Thus $\mathsf{c}_0 \in \textup{Port}_{III}({\boldsymbol{\mathsf{c}}})$. By
Condition 1, $\tau(1)=1-H$. So we have $\Ret_{\boldsymbol{\mathsf{c}}}(\mathsf{c}_1) = \mathsf{c}_0 = \mathsf{c}_{\tau(1)} $.
Also if $\esc(0) = 2$, then $\mathsf{c}_0$ has at least two distinct non-critical children with $N_1(\mathsf{v}_0^{\child_i}) = 1
= N^1_{\boldsymbol{\mathsf{c}}}(\mathsf{c}_0)$ so it is a compound $\boldsymbol{\mathsf{c}}$-portal. We satisfy all hypotheses
stated above, so we can start the inductive process.

Suppose that we have constructed $\tree_1, \dots, \tree_L$ satisfying the above hypotheses for some $L \geq 2$.  We apply Lemma
\ref{lem: ext} and extend to length $L+1$. If $L$ is a $\tau$-portal, then $N^1_{\boldsymbol{\mathsf{c}}}(\mathsf{c}_L) <
N^1_{\boldsymbol{\mathsf{c}}}(\mathsf{c}_{L+1})$ and $\mathsf{c}_L$ is a compound $\boldsymbol{\mathsf{c}}$-portal by Lemma
\ref{cor: classify type I c-portals}.  If $L \in E$, then $D_L - D_{L+1} \geq \esc(L)$ so there is an escape at $\mathsf{c}_{L}$
provided $\esc(L) >0$. Thus $\mathsf{c}_{L} \in \textup{Port}_{III}({\boldsymbol{\mathsf{c}}})$ by Lemma \ref{lem: Port III
categorized}. If $\esc(L) = 2$, then $\mathsf{c}_{L}$ is a compound $\boldsymbol{\mathsf{c}}$-portal by Lemma \ref{lem: number of
n.c.c of type 3 c-portal}.

\end{proof}

\begin{figure}[hbt]  \label{fig: top of tree}
 \begin{center}
    \includegraphics{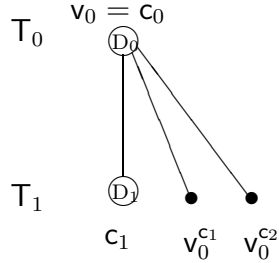}
    \caption{The top of $\tree$ from the proof of Proposition \ref{prop: tau TwD} when $D_0 -D_1 =2$.}
 \end{center}
\end{figure}

\subsection{Proof of the Extension Lemma} \label{subsect: Proof of ext lem}

We now prove the Extension Lemma. Suppose that $\tau$ is an abstract Yoccoz return function with exceptional set $E$ and
$(D_l)_{l \in \bZ}$ is a $\tau$-admissible sequence.  We say a tree with dynamics $(\tree, \F)$ of length $L$ that satisfies the
hypotheses of the Extension Lemma \emph{realizes $\tau$ to level $L$}.  We prove the Extension Lemma by showing that we can find
an extension of $(\tree, \F)$ of length $L+1$ that realizes $\tau$ to level $L+1$.  The only difficulty is to extend in such a
way that $\Ret_{\boldsymbol{\mathsf{c}}}(\mathsf{c}_{L+1}) = \mathsf{c}_{\tau(L+1)}$ (equivalently
$N^1_{\boldsymbol{\mathsf{c}}}(\mathsf{c}_{L+1}) = N^{R(L+1)}_{\boldsymbol{\mathsf{c}}}(\mathsf{c}_{L})$). The proof consists of
considering various cases for $R(L+1)$.  The cases when $R(L+1) = 1$ or $2$ are straightforward.  If $R(L+1) \geq 3$, we need to
carefully examine $\tau^{R(L+1) - 1}(L)$.  In particular, whether or not it is a simple or compound
$\boldsymbol{\mathsf{c}}$-portal.  We prove the lemma when we have covered all possible cases.

\begin{lem} \label{lem: ext choose FN(vc)}
Let $(\tree, \F)$ be a tree with dynamics of length $L$.  Suppose that $\mathsf{v} \in\tree_L$, $D\in \bZ$ with $1 \leq D \leq
\deg  \mathsf{v} $, and $\F^N(\mathsf{v})= \mathsf{w}$ for some $N \geq 1$. If $\mathsf{w}^{\child_0}$ is a child of  $\mathsf{w}
$, then there exists $(\tree',\F')$, an extension of $(\tree,\F)$ of length $L+1$ with $\mathsf{v}^{\child_0} \in \tree_{L+1}'$
such that
\begin{enumerate}
    \item $\mathsf{v}^{\child_0}$ is a child of $\mathsf{v}$,
    \item $\deg \mathsf{v}^{\child_0} = D$,
    \item $\F^N(\mathsf{v}^{\child_0}) = \mathsf{w}^{\child_0}$.
\end{enumerate}
\begin{proof}
The hard case is when $N=1$, the general case follows by a straightforward induction argument. We define $\tree_{L+1}'$ in three
steps. First let $\mathsf{v}^{\child_0} \in \tree_{L+1}'$ have the above 3 properties.  We need to give $\mathsf{v}$ enough
additional children to satisfy D2, and be sure we do not give $\mathsf{v}$ so many children that we violate D1.  Since $\deg
\mathsf{v}^{\child_0} = D \leq \deg \mathsf{v}$, we can give $\mathsf{v}$ exactly $\deg  \mathsf{v} - D$ non-critical children
such that $\F(\mathsf{v}^{\child}) = \F(\mathsf{v})^{\child_0}$. For each other child of $\F(\mathsf{v})$,
$\F(\mathsf{v})^{\child_i} \neq \F(\mathsf{v})^{\child_0}$, $\mathsf{v}$ gets $\deg  \mathsf{v}$ non-critical children such that
$\F(\mathsf{v}^{\child}) = \F(\mathsf{v})^{\child_i}$. For any other $\mathsf{u} \in \tree_L$ we give $\mathsf{u}$ exactly $\deg
\mathsf{u}$ non-critical children which map to each child of $\F(\mathsf{u})$.

\end{proof}
\end{lem}

It is easy to extend in such a way that the first return time of the child is the first return time of the parent.

\begin{cor} \label{cor: ext with N_1 equal}
Let $(\tree, \F)$ be a tree with dynamics of length $L$. Suppose that $\mathsf{v} \in\tree_L$ and $D\in \bZ$ with $1 \leq D \leq
\deg \mathsf{v} $.  If $(\tree, \F)$ has a critical branch ${\boldsymbol{\mathsf{c}}} = \set{\mathsf{c}_0, \dots, \mathsf{c}_L}$,
then there exists $(\tree',\F')$, an extension of $(\tree,\F)$ of length $L+1$ with $\mathsf{v}^{\child_0} \in \tree_{L+1}'$ such
that $\mathsf{v}^{\child_0}$ is a child of $\mathsf{v}$, $\deg \mathsf{v}^{\child_0} = D$, and
$N^{\boldsymbol{\mathsf{c}}}_1(\mathsf{v}^{\child_0}) = N^{\boldsymbol{\mathsf{c}}}_1(\mathsf{v})$.

\begin{proof}
Say that $\Ret_{\boldsymbol{\mathsf{c}}}(\mathsf{v}) = \mathsf{c}_l$ for some $l <L$. Then by assumption, $\mathsf{c}_{l+1}$ is a
vertex of $\tree$ and a child of $\mathsf{c}_l$ since $ 0 \leq l < L$. Apply Lemma \ref{lem: ext choose FN(vc)} with $\mathsf{w}
= \mathsf{c}_l$, $\mathsf{w}^{\child_0} = \mathsf{c}_{l+1}$, and $N = N^{\boldsymbol{\mathsf{c}}}_1(\mathsf{v})$. The first two
conclusion follow immediately. Also we have $N^{\boldsymbol{\mathsf{c}}}_1(\mathsf{v}^{\child_0}) \leq
N^1_{\boldsymbol{\mathsf{c}}}(\mathsf{v})$ since $\mathsf{c}_{l+1} \in {\boldsymbol{\mathsf{c}}} $. We always have
$N_{\boldsymbol{\mathsf{c}}}^1(\mathsf{v}^{\child_0}) \geq N_{\boldsymbol{\mathsf{c}}}^1(\mathsf{v})$ by Corollary \ref{cor: N_1
non-dec }. Thus the desired equality is true.

\end{proof}
\end{cor}

We consider extension when $R(L+1) = 1$.

\begin{cor} \label{cor: R=1 => ext}
Suppose that $(\tree, \F)$ realizes $\tau$ to level $L$. If $R(L+1) = 1$, then there is an extension of $(\tree, \F)$ of length
$L+1$ that realizes $\tau$ to level $L+1$.

\begin{proof}
Apply the above corollary with $\mathsf{v}= \mathsf{c}_L$, $\mathsf{v}^{\child_0}= \mathsf{c}_{L+1}$, and $D = D_{L+1}$.
\end{proof}
\end{cor}

We also prove the case when $R(L+1) = 2$.

\begin{sublem} \label{sublem: R=2 => ext}
Suppose that $(\tree, \F)$ realizes $\tau$ to level $L$. If $R(L+1) = 2$, then there is an extension of $(\tree, \F)$ of length
$L+1$ that realizes $\tau$ to level $L+1$.
\begin{proof}
Say $\tau(L) =m$.  By Lemma \ref{main lem}, $\mathsf{c}_m $ is a $\boldsymbol{\mathsf{c}}$-portal.  Say $\mathsf{c}_m^{\child_0}$
is a non-critical child of $\mathsf{c}_m$ such that $N_{\boldsymbol{\mathsf{c}}}^1(\mathsf{c}_m ) =
N_{\boldsymbol{\mathsf{c}}}^1(\mathsf{c}_m^{\child_0} )$. Apply Lemma \ref{lem: ext choose FN(vc)} with $\mathsf{v} =
\mathsf{c}_L$, $D = D_{L+1}$, $N = N^{\boldsymbol{\mathsf{c}}}_1(\mathsf{c}_L)$, $mathsf{w} =\mathsf{c}_m  $, and
$\mathsf{w}^{\child_0} = \mathsf{c}_m^{\child_0}$. Then $\F^N(\mathsf{c}_{L+1})=\mathsf{c}_m^{\child_0} $. Since
$N_{\boldsymbol{\mathsf{c}}}^1(\mathsf{c}_{L+1}) \geq N$ and $\mathsf{c}_m^{\child_0}$ is non-critical, we have
$N_{\boldsymbol{\mathsf{c}}}^1(\mathsf{c}_{L+1}) = N + N_{\boldsymbol{\mathsf{c}}}^1(\mathsf{c}_m^{\child_0} ) =
N^2_{\boldsymbol{\mathsf{c}}}(\mathsf{c}_{L})$ by Lemma \ref{lem: ret times add}.
\end{proof}

\end{sublem}

It is harder to show that there is an extension for the case $R(L+1) \geq 3$. We do so not just keeping track of
$\boldsymbol{\mathsf{c}}$-portals, but by considering simple versus compound $\boldsymbol{\mathsf{c}}$-portals. Particularly we
need conditions that insure that a vertex does not obstruct a portal.

We give a general result about the dynamics from the children of one vertex to the children of an iterate of the vertex.

\begin{lem}\label{lem: 2c of FN(v) hit by 2 ncc of v}
Let $(\tree, \F)$ be a tree with dynamics possibly of finite length.  Suppose that $\mathsf{v} \in\tree$, and $\F^N(\mathsf{v})=
\mathsf{w}$ for some $N \geq 1$. If $\mathsf{w}^{\child_1}$ and $\mathsf{w}^{\child_2}$ are two distinct children of $\mathsf{w}
$, then $\mathsf{v}$ has at least two non-critical children, $\mathsf{v}^{\child_1}$ and $\mathsf{v}^{\child_2}$, such that
\[
    \mathsf{v}^{\child_i} \in \F^{-N}(\set{\mathsf{w}^{\child_1}, \mathsf{w}^{\child_2}}) \quad i = 1,2.
\]

\begin{proof}
We proceed by induction on $N$.  The hard case is when $N=1$, the general case is straightforward. Let
$\F^{-1}(\set{\mathsf{w}^{\child_1}, \mathsf{w}^{\child_2}})  = \set{\mathsf{v}^{\child_1}, \dots, \mathsf{v}^{\child_k}}$ for
some $k$. By D1,
\begin{align*}
    \sum_{i=1}^k (\deg \mathsf{v}^{\child_i}- 1) &\leq \deg \mathsf{v} -1\\
    \left( \sum_{i=1}^k \deg \mathsf{v}^{\child_i} \right)
    - k &\leq \deg \mathsf{v} -1.\\
    \intertext{Now $ \displaystyle \sum_{i=1}^k \deg \mathsf{v}^{\child_i} = 2 \deg \mathsf{v}$ by D2. So}
    2 \deg \mathsf{v} - k &\leq \deg \mathsf{v} -1\\
    \deg \mathsf{v} + 1 &\leq k.
\end{align*}
It follows from this lower bound on $k$ and the second inequality above that at most $k-2$ of $\set{\mathsf{v}^{\child_1}, \dots,
\mathsf{v}^{\child_k}}$ are critical.
\end{proof}
\end{lem}

In the above lemma, the assumption that $\mathsf{w}$ has two distinguished children cannot be weakened. If $\mathsf{w}$ has only
one distinguished child, then it is quite easy to construct examples where $\mathsf{v}$ has only one child mapped to the
distinguished child. We apply the above lemma in the case when $\mathsf{w}$ is a compound $\boldsymbol{\mathsf{c}}$-portal.

\begin{cor} \label{cor: ret to compound c-port => 2 ncc}
Let $(\tree, \F)$ be a tree with dynamics possibly of finite length.  Let $\boldsymbol{\mathsf{c}}$ be a critical end or branch.
Let $\mathsf{v} \in \tree$  and $ \mathsf{c}_k \in\boldsymbol{\mathsf{c}} $ with $\Ret^S_{\boldsymbol{\mathsf{c}}}(\mathsf{v})=
\mathsf{c}_k$ for some $S \geq 1$. If $\mathsf{c}_k $ is a compound $\boldsymbol{\mathsf{c}}$-portal, then $\mathsf{v}$ has two
distinct non-critical children $\mathsf{v}^{\child_1}$ and $\mathsf{v}^{\child_2}$, such that
$N^1_{\boldsymbol{\mathsf{c}}}(\mathsf{v}^{\child_i}) = N^S_{\boldsymbol{\mathsf{c}}}(\mathsf{v}) +
N^1_{\boldsymbol{\mathsf{c}}}(\mathsf{c}_k)$ for $i=1,2$.

\begin{proof}
We will use induction on $S$. The induction step is easy, so we only give the proof when $S=1$. Since $\mathsf{c}_k $ is a
compound $\boldsymbol{\mathsf{c}}$-portal, it has two children $\mathsf{c}_k^{\child_1}, \mathsf{c}_k^{\child_2} \notin
\boldsymbol{\mathsf{c}}$ with $N^1_{\boldsymbol{\mathsf{c}}}(\mathsf{c}_k) =
N^1_{\boldsymbol{\mathsf{c}}}(\mathsf{c}_k^{\child_j})$ for $j=1,2$. Apply Lemma \ref{lem: 2c of FN(v) hit by 2 ncc of v} with
$\mathsf{w} = \mathsf{c}_k$, $\mathsf{w}^{\child_j} = \mathsf{c}_k^{\child_j}$, and $N =
N^1_{\boldsymbol{\mathsf{c}}}(\mathsf{v})$. So $\mathsf{v}$ has at least two non-critical children
$\mathsf{v}^{\child_1},\mathsf{v}^{\child_2} \in \F^{-N}(\set{\mathsf{c}_k^{\child_1}, \mathsf{c}_k^{\child_2}})$. Fix $i=1$ or
$2$. By Corollary \ref{cor: N_1 non-dec }, $N^1_{\boldsymbol{\mathsf{c}}}(\mathsf{v}^{\child_i}) \geq N$. But
$\F^N(\mathsf{v}^{\child_i}) = \mathsf{c}_k^{\child_j}$ for some $j$.  So $\F^N(\mathsf{v}^{\child_i}) \notin
\boldsymbol{\mathsf{c}} $ by definition of $ \mathsf{c}_k^{\child_j} $. Therefore by Lemma \ref{lem: ret times add},
$N^1_{\boldsymbol{\mathsf{c}}}(\mathsf{v}^{\child_i}) = N + N^1_{\boldsymbol{\mathsf{c}}}(\mathsf{c}_k^{\child_j}) =
N^1_{\boldsymbol{\mathsf{c}}}(\mathsf{v}) + N^1_{\boldsymbol{\mathsf{c}}}(\mathsf{c}_k)$.

\end{proof}
\end{cor}

We prove another case of the Extension Lemma where we pass through a compound $\boldsymbol{\mathsf{c}}$-portal.

\begin{sublem} \label{sublem: R geq 3 and compound => ext}
Suppose that $(\tree, \F)$ realizes $\tau$ to level $L$.  Suppose $R(L+1) \geq 3$ and $\tau^{R-1}(L) =m$. If $\mathsf{c}_m$ is a
compound $\boldsymbol{\mathsf{c}}$-portal, then there is an extension of $(\tree, \F)$ of length $L+1$ that realizes $\tau$ to
level $L+1$.
\begin{proof}
Let $R = R(L+1)$. Apply Corollary \ref{cor: ret to compound c-port => 2 ncc} with $\mathsf{v}=
\Ret_{\boldsymbol{\mathsf{c}}}(\mathsf{c}_L)$ and $S = R-2$, to get a non-critical child $\mathsf{v}^{\child_1}$ with $
N^1_{\boldsymbol{\mathsf{c}}}(\mathsf{v}^{\child_1}) = N_{R-2}^{\boldsymbol{\mathsf{c}}}(\mathsf{v}) +
N^1_{\boldsymbol{\mathsf{c}}}(\mathsf{c}_m)$. By Lemma \ref{lem: ext choose FN(vc)}, we can extend to a tree with dynamics of
length $L+1$ with $\mathsf{c}_{L+1} \in \tree_{L+1}'$ such that

\begin{enumerate}
    \item $\mathsf{c}_{L+1}$ is a child of $\mathsf{c}_{L}$,
    \item $\deg \mathsf{c}_{L+1} = D_{L+1}$,
    \item $\F^{N^1_{\boldsymbol{\mathsf{c}}}(\mathsf{c}_L)}(\mathsf{c}_{L+1}) = \mathsf{v}^{\child_1}$.
\end{enumerate}

Since $\mathsf{v}^{\child_1}$ is non-critical, Lemma \ref{lem: ret times add} implies
\begin{align*}
    N^1_{\boldsymbol{\mathsf{c}}}(\mathsf{c}_{L+1}) &= N^1_{\boldsymbol{\mathsf{c}}}(\mathsf{c}_L)+
        N^1_{\boldsymbol{\mathsf{c}}}(\mathsf{v}^{\child_1})\\
             &= N^1_{\boldsymbol{\mathsf{c}}}(\mathsf{c}_L) + N_{R-2}^{\boldsymbol{\mathsf{c}}}(\mathsf{v}) +
        N^1_{\boldsymbol{\mathsf{c}}}(\mathsf{c}_m)\\
             &= N^1_{\boldsymbol{\mathsf{c}}}(\mathsf{c}_L) +
        N_{R-2}^{\boldsymbol{\mathsf{c}}}(\Ret_{\boldsymbol{\mathsf{c}}}(\mathsf{c}_L)) +
        N_{1}^{\boldsymbol{\mathsf{c}}}(\Ret^{R-1}_{\boldsymbol{\mathsf{c}}}(\mathsf{c}_L))\\
            &= N^R_{\boldsymbol{\mathsf{c}}}(\mathsf{c}_L).
\end{align*}

\end{proof}
\end{sublem}

We prove the case of the Extension Lemma when $R(L+1) \geq 3$ and $\mathsf{c}_L$ passes through a simple
$\boldsymbol{\mathsf{c}}$-portal $\mathsf{c}_m$. In light of Lemma \ref{lem: Returns to simple portals} we need to ensure that
$\mathsf{c}_m$ is not obstructed by some other vertex.

\begin{sublem} \label{sublem: R geq 3 and simple => Ext}
Suppose that $(\tree, \F)$ realizes $\tau$ to level $L$. Let $R= R(L+1) \geq 3$ and $\tau^{R-1}(L) =m$. If $\mathsf{c}_m$ is a
simple $\boldsymbol{\mathsf{c}}$-portal, then there is an extension of $(\tree, \F)$ of length $L+1$ that realizes $\tau$ to
level $L+1$.
\begin{proof}
We must have $\esc(m) = 1$ or else $\mathsf{c}_m$ would be a compound $\boldsymbol{\mathsf{c}}$-portal by Lemma \ref{lem: number
of n.c.c of type 3 c-portal}.  Let $\tau^{R-2}(L) = k$. We cannot have $R(k+1) = 2$ or we have $\esc(m)=2$ by Definition
\ref{defn: esc(m)}.2.  So $R(k+1) \neq 2$, which implies $\tau(k+1) \neq \tau(L+1)$.  Thus $\mathsf{c}_k$ has at least two
non-critical children $\mathsf{c}_k^{\child_i}$ with $\Ret(\mathsf{c}_k^{\child_i})= \mathsf{c}_{\tau(L+1)}$. Apply Corollary
\ref{cor: ret to compound c-port => 2 ncc}, to get an extension with $\Ret_{\boldsymbol{\mathsf{c}}} (\mathsf{c}_{L+1}) =
\Ret_{\boldsymbol{\mathsf{c}}}(\mathsf{c}_k^{\child_i})$ for some $i$.

\end{proof}
\end{sublem}

Finally we prove the Extension Lemma.

\begin{proof}[Proof of Lemma \ref{lem: ext}.]
Let $R = R(L+1)$. We want an extension of $(\tree, \F)$ such that $N^1_{\boldsymbol{\mathsf{c}}}(\mathsf{c}_{L+1})
=N^R_{\boldsymbol{\mathsf{c}}}(\mathsf{c}_{L}) $.  If $R=1$, then we can apply Corollary \ref{cor: ext with N_1 equal}. If $R=
2$, then we can apply Sublemma \ref{sublem: R=2 => ext}.

It remains to show the case when $R \geq 3$. Say that $\tau^{R-1}(L) = m$, so
$\Ret^{R-1}_{\boldsymbol{\mathsf{c}}}(\mathsf{c}_{L}) = \mathsf{c}_{m}$. By Condition 3 of Theorem \ref{thm: tau cond}, either
$\tau^{R-1}(L)$ is a $\tau$-portal or $\tau^{R-1}(L) \in E$.  First consider the case when $\tau^{R-1}(L)= m$ is a $\tau$-portal.
By Lemma \ref{lem: tau-portal equivalences}, $N^1_{\boldsymbol{\mathsf{c}}}(\mathsf{c}_{m}) <
N^1_{\boldsymbol{\mathsf{c}}}(\mathsf{c}_{m+1})$. Thus $\mathsf{c}_{m}$ is a compound $\boldsymbol{\mathsf{c}}$-portal by
Corollary \ref{cor: classify type I c-portals}.  So Sublemma \ref{sublem: R geq 3 and compound => ext} applies and we can extend.
Next assume that $\tau^{R-1}(L)= m \in E$. Since $(D_l)$ is $\tau$ admissible, we have $D_m -D_{m-1} \geq \esc(m) \geq 1$. Thus
$\mathsf{c}_{m} \in \textup{Port}_{III}(\boldsymbol{\mathsf{c}})$. If $\mathsf{c}_{m}$ is a compound
$\boldsymbol{\mathsf{c}}$-portal, then we can extend using Sublemma \ref{sublem: R geq 3 and compound => ext}.  If
$\mathsf{c}_{m}$ is a simple $\boldsymbol{\mathsf{c}}$-portal, then an application of Sublemma \ref{sublem: R geq 3 and simple =>
Ext} finishes the proof.

\end{proof}

We remark that the tree with dynamics that we construct using Proposition \ref{prop: tau TwD} has no type~II
$\boldsymbol{\mathsf{c}}$-portals. It follows from Lemma \ref{lem: tau simple portals} that $\esc(\tau)+2$ is the minimal degree
for a polynomial without type~II $\boldsymbol{\mathsf{c}}$-portals that realizes $\tau$.  We might reduce the number of escaping
critical points and thereby reduce the degree of the polynomial in some cases if in our construction some points in $E$
corresponded to type~II $\boldsymbol{\mathsf{c}}$-portals instead of type~III $\boldsymbol{\mathsf{c}}$-portals. However using
type~II $\boldsymbol{\mathsf{c}}$-portals would require analyzing the returns of two ends to each other, which is a question
beyond the scope of this paper.

\end{document}